\RequirePackage{fix-cm}
\documentclass[smallextended]{svjour3}

\usepackage{mathptmx}      
\smartqed  

\def\makeheadbox{}

\usepackage{amsmath,amssymb,wasysym,bbm,url,theorem,hyperref}
\usepackage[polutonikogreek,english]{babel}
\usepackage[utf8x]{inputenx}

\newcommand{\greek}[1]{{\selectlanguage{polutonikogreek}#1}}

\let\phi\varphi
\let\epsilon\varepsilon
\let\theta\vartheta
\let\mathbb\mathbbm

\newcommand{\N}{\mathbb{N}}

\newcommand{\R}{\mathbb{R}}
\newcommand{\C}{\mathbb{C}}

\newcommand{\nnz}{^\times}
\newcommand{\der}{\partial}
\newcommand{\evl}{\text{\scshape\texttt e}}
\newcommand{\cum}{{\textstyle \varint}}
\newcommand{\ocum}{{\setbox0=%
    \hbox{$\textstyle{\scriptstyle-}{\varint}$}%
    \textstyle{\vcenter{\hbox{$\scriptstyle-$}}\kern-.5\wd0}%
    \!\varint}}

\newcommand{\exppoly}{K[x,e^x]}

\newcommand{\galg}{\mathcal{F}}
\newcommand{\calg}{\mathcal{E}}
\newcommand{\bspc}{\mathcal{B}}
\newcommand{\cspc}{\mathcal{C}}
\newcommand{\fspc}{\mathcal{G}}
\newcommand{\const}{\mathcal{C}}
\newcommand{\init}{\mathcal{I}}
\newcommand{\diffop}{\calg[\der]}
\newcommand{\gdiffop}{\galg[\der]}
\newcommand{\intdiffop}[1][]{\galg_{#1}[\der,\cum]}
\newcommand{\gintop}{\galg[\cum]}

\newcommand{\inner}[2]{\langle #1 | #2 \rangle}

\newcommand{\bspclat}{K\Phi}
\newcommand{\allprob}{\calg[\der] \ltimes K\Phi}
\newcommand{\regprob}{\calg[\der]_{\Phi}}
\newcommand{\locprob}{K\calg[\der]_{\Phi}^\star}
\newcommand{\hypfun}[2][\Phi]{#2_{#1}}
\newcommand{\locfun}[2][\Phi]{#2_{#1}^\star}
\newcommand{\greenmon}{\mathfrak{G}}
\newcommand{\locgreen}[1][]{\greenmon^\star}

\newcommand{\loc}[2]{#1^{-1} \! #2}
\newcommand{\rloc}[2]{#1 #2^{-1}}

\newcommand{\Ker}[1]{\operatorname{Ker}(#1)}
\newcommand{\Img}[1]{\operatorname{Im}(#1)}

\newcommand{\orth}[1]{#1^\perp}

\newcommand{\dirs}{\dotplus}
\newcommand{\bvp}[2]{\boxed{\begin{array}{l}#1\\#2\end{array}}}
\newcommand{\fri}[1]{#1^\Diamond}

\newcommand{\shift}[1]{S_{#1}}

\newcommand{\fafac}[2]{#1^{\underline{#2}}}
\newcommand{\superfac}[1]{\operatorname{sf}(#1)}

\DeclareMathOperator{\lclm}{lclm}
\DeclareMathOperator{\ord}{ord}

\providecommand{\abs}[1]{\lvert#1\rvert}
\providecommand{\norm}[1]{\lVert#1\rVert}

\reversemarginpar

\newcommand{\tma}{TH$\exists$OREM$\forall$}
\newcommand{\mpl}{\textsc{Maple}}
\newcommand{\mma}{\textsc{Mathematica}}

\begin{document}

\title{A Noncommutative Mikusi\'{n}ski Calculus\thanks{The first
    author acknowledges support through the EPSRC First Grant
    EP/I037474/1.}}  \author{Markus Rosenkranz \and Anja Korporal}

\institute{M. Rosenkranz \at
  School of Mathematics, Statistics and Actuarial Science,
  University of Kent, Canterbury CT2 7NF,\\
  Tel.: +44 1227 82-3650,
  Fax: +44 1227 82-7932,
  \email{M.Rosenkranz@kent.ac.uk}
  \and
  A. Korporal \at
  Research Institute for Symbolic Computation,
  Johannes Kepler University, 4040 Linz, Austria
}

\date{Submitted \today}

\maketitle

\begin{abstract}
  We set up a left ring of fractions over a certain ring of boundary
  problems for linear ordinary differential equations. The fraction
  ring acts naturally on a new module of generalized functions. The
  latter includes an isomorphic copy of the differential algebra
  underlying the given ring of boundary problems. Our methodology
  employs noncommutative localization in the theory of
  integro-differential algebras and operators. The resulting
  structure allows to build a symbolic calculus in the style of
  Heaviside and Mikusi{\'n}ski, but with the added benefit of
  incorporating boundary conditions where the traditional calculi
  allow only initial conditions.  
\end{abstract}

\section{Introduction}
\label{sec:introduction}

\paragraph{General Context.} Linear boundary problems are a crucial
concern of applied
mathematics~\cite{Duffy2001,Agarwal1986,AgarwalORegan2008,Stakgold1979}. One
might thus expect a rich \emph{algebraic theory} with \emph{symbolic
  algorithms}, so as to support the exact solution and manipulation of
a suitable class of boundary problems. Alas, this is not the case yet.

There may be two reasons for this. One is that the classical algebraic
theory of (nonlinear) differential equations, the \emph{differential
  algebra} built up by Ritt~\cite{Ritt1966} and
Kolchin~\cite{Kolchin1973}, does not lend itself easily to boundary
conditions: The elements of a differential ring or field are not
functions but abstract objects that cannot be evaluated at a boundary
point.

The other reason is of a pragmatic nature. If a differential equation
can be solved at all, one tends to delegate boundary conditions to
\emph{adhoc postprocessing} steps that would adapt the integration
constants/functions of the ``general solution''. While this may be
viable for linear ordinary differential equations (LODEs), the notion
of general solution is much less useful in the case of partial
differential equations (LPDEs): For example, there is not much point
in solving the $u_{xx} + u_{yy} = 0$ per se, but there are useful
representations of those solutions that satisfy Dirichlet boundary
conditions on the unit disc (Poisson kernel). So even though we will
be dealing only with LODEs in this paper, the larger context of LPDEs
should be kept in mind.

In the admittedly modest case of \emph{boundary problems for LODEs},
an algebraic theory was set up in~\cite{Rosenkranz2005}. The decisive
step was to expand the well-known structure of a differential ring,
complementing its derivation by a compatible Baxter operator (integral
operator). This solves at once two problems: It provides the algebraic
structure needed for expressing Green's operators (solution operators
for boundary problems), and it yields an evaluation (a multiplicative
functional). The notion of evaluation is the key for imposing boundary
conditions on the otherwise abstract objects of a differential ring;
one supplies as many evaluation as needed (in addition to the one
coming from the Baxter operator).

In~\cite{RosenkranzRegensburger2008}, the basic theory
of~\cite{Rosenkranz2005} was refined and generalized. Moreover, the
authors have introduced a multiplicative structure on boundary
problems that will also be crucial for our buildup of the Heaviside
calculus. The multiplication of boundary problems corresponds to the
composition of their Green's operators in reverse order; see
Equation~\eqref{eq:anti-isomorphism}.

The framework of boundary problems
of~\cite{Rosenkranz2005,RosenkranzRegensburger2008} was
\emph{implemented} several times: The first implementation was coded
in \mma/\tma\ as an external package for boundary problems with
constant coefficients. This version was superseded by a new
implementation as an internal \tma\ functor for generic
integro-differential algebras
in~\cite{RosenkranzRegensburgerTecBuchberger2009,Tec2011,%
  RosenkranzRegensburgerTecBuchberger2012}. Now in the third
generation, the most recent implementation is the Maple
package~\emph{IntDiffOp}~\cite{KorporalRegensburgerRosenkranz2010,%
  KorporalRegensburgerRosenkranz2011,KorporalRegensburgerRosenkranz2012},
which provides additional support for singular boundary problems (all
previous packages being restricted to regular boundary problems).

\paragraph{Heaviside Calculus.} Our main concern in this paper is to
build a bridge from the framework of symbolic boundary problems to the
classical Heaviside
calculus~\cite{Heaviside1893,Heaviside1894,Mikusinski1959}.  Since
O.~Heaviside's main idea, subsequently made rigorous by
J.~Mikusi\'{n}ski, was to treat the differential operator as a
\emph{symbolic multiplier}, it is perhaps natural to reflect on
possible connections to contemporary symbolic methods for differential
equations. But why to symbolic boundary problems?

The answer lies in Heaviside's so-called \emph{fundamental formula},
the algebraic analog of a well-known relation for the Laplace
transform. If~$s$ is the ``symbolic multiplier'' representing the
differential operator and~$f$ is a suitable function, then one has~$sf
= f' + f(0) \, \delta_0$. The point is that~$s$, which is supposed to
be invertible in this calculus, must somehow ``remember'' the
integration constant that is lost in differentiation. We will come
back to this point in more detail
(Section~\ref{sec:methorious-functions}). At this point it suffices to
say that an appropriate action of~$s$ on functions involves an
evaluation operator. It is used in the Heaviside-Mikusinski calculus
for incorporating the initial values associated to a differential
equation, thus yielding at once the solution of the whole initial
value problem. It is thus natural---staying in the frame of LODEs---to
ask if one can build up a more general calculus that would allow to
incorporate multiple boundary values, given by several evaluation
operators.

The basic ideas of Heaviside and Mikusi\'{n}ski have been vastly
extended, specifically by L. Berg~\cite{Berg1962} and
I. Dimovski~\cite{Dimovski1990}. The latter has also broached a
question closely related to the one raised above, namely the setup of
\emph{nonlocal convolutions}~\cite{Dimovski1994,Spiridonova2010} and
\emph{custom-tailored convolutions} for boundary
problems~\cite{Dimovski2012}. As far as we know, however, there are no
direct generalizations of the fundamental formula from one to more
evaluation operators acting on univariate functions.

\paragraph{Basic setup.} Our own approach is different in many
respects. It does not qualify as an \emph{operational calculus} in
Dimovski's understanding~\cite{Dimovski2012}, for whom its crucial
feature is that ``operators'' and ``operands'' are merged in a single
data structure, whereas algebraic analysis keeps the module of operand
functions separate from the ring of operators. In this sense, we
follow the line of algebraic analysis since we distinguish the ring of
methorious operators (Section~\ref{sec:localization}) from the module
of methorious functions
(Section~\ref{sec:methorious-functions}). Moreover, our approach is
genuinly algebraic while the hallmark of the Heaviside-Mikusi\'{n}ski
tradition is an ingenious mix of algebra and analysis. Nevertheless,
we believe that our setup is close in spirit to the original
Heaviside-Mikusi\'{n}ski setup, enjoying the following attractive
features:
\begin{itemize}
\item Its \emph{basic philosophy} follows closely in Heaviside's
  footsteps, making the differential operator invertible by
  ``remembering'' suitable boundary data.
\item It provides an algebraic structure that accommodates boundary
  problems for an \emph{arbitrary number of evaluation points} as well
  as \emph{nonlocal conditions}.
\item All boundary problems are covered \emph{uniformly}, so one need
  not set up custom-tailored multiplications (convolutions) for each
  type of boundary conditions one wants to consider.
\item It illuminates the passage from one to several evaluations
  algebraically: Localization takes place in a \emph{noncommutative
    ring}. In contrast, convolution algebras are commutative by
  definition~\cite{Dimovski1990}.
\item The construction is \emph{generic in the coefficient algebra}
  (it works for the class of umbral integro-differential algebras).
\end{itemize}
From the viewpoint of analysis, however, the class of umbral
coefficient algebras is rather limited compared to Mikusi{\'n}ski's
setup (continuous or even $L^2$ functions on the positive
half-axis). While this might be relaxed by suitable limit
considerations, this is not in the interest of the present paper,
where we want to focus on the algebraic aspects.

Our goal in this paper is to build a \emph{first bridge} between the
algebraic theory of boundary
problems~\cite{Rosenkranz2005,RosenkranzRegensburger2008} and
Heaviside's tactic of using ``symbolic fractions'' (localization in a
suitable convolution ring) for integrating differential equations with
initial/boundary data. In fact, it gives a new justification (see
after Proposition~\ref{prop:greensop-localization}) for a notational
device initiated in~\cite{RosenkranzRegensburger2008}: Since then we
have written~$\mathcal{B}^{-1}$ for the Green's operator associated to
a boundary problem~$\mathcal{B}$, originally chosen in view of the
anti-isomorphism~\eqref{eq:anti-isomorphism} mentioned above.

The work in this paper may also be seen as an answer to an (implicit)
question originating
from~\cite{RegensburgerRosenkranzMiddeke2009}. The setting there was
restricted to $K[x]$ coefficients and treated via Ore algebras. The
object of central interest was the \emph{integro-differential Weyl
  algebra}, which has shown to have two important quotient algebras:
The ring of localized differential operators~$K[\der, \der^{-1}][x]$
and the usual ring of integro-differential operators~$K[x][\der,
\cum]$. The former has $\der$ as a two-sided inverse, but no action
on~$K[x]$; the latter has the action, but~$\cum$ is only a right
inverse of~$\der$. Now the question arises: Can one build up a
structure, with an action on~$K[x]$ or other coefficient algebras,
such that the derivation has a two-sided inverse? We show in this
paper that the answer is affirmative, provided the derivation is
enhanced in a way similar to Heaviside's symbolic multiplier~$s$.

As always in algebra, localizing a ring sheds new light on its
structure, especially in the noncommutative setting (here the theory
of fraction rings is somewhat more delicate due to the Ore condition,
see Section~\ref{sec:localization}). We hope this will also be the
case for our present construction. But one must bear in mind that the
subject has only been touched and various issues remain in a
preliminary and unsatisfactory state (see the
Conclusion). Nevertheless, a wealth of new relations has been
uncovered, and we are confident that they will allow interesting
generalizations and refinements.

\paragraph{Structure of the paper.} We start out by summarizing the
basic theory of symbolic boundary problems, giving special emphasis to
their monoid structure (Section~\ref{sec:monoid-bp}). Our construction
is based on a certain kind of boundary conditions, which we have
called ``umbral'' because of certain relations to the umbral calculus
(Section~\ref{sec:umbral-bc}). After this preparation, we tackle the
task of localization for an integro-differential algebra with an
umbral character set (Section~\ref{sec:localization}). For the
resulting ring of fractions, we construct a module of functions on
which it acts naturally (Section~\ref{sec:methorious-functions}). We
conclude with some remarks about possible extensions and
generalizations (Section~\ref{sec:conclusion}).

\paragraph{Notation.} All rings are with unit but not necessarily
commutative. A domain is a (commutative or noncommutative) ring
without zero divisors. The zero vector space of any dimension will be
denoted by~$O$. We use the notation $\galg_1 \le \galg_2$ for
indicating that a vector space~$\galg_1$ is a subspace of a vector
space~$\galg_2$.

\section{The Monoid of Boundary Problems}
\label{sec:monoid-bp}

Our basic setting is that of~\cite{RosenkranzRegensburger2008a}. We
review the main results here for making our present treatment more
self-contained and for introducing various pieces of notation and
terminology in their proper places. We start from the notion of
\emph{integro-differential algebra}, a natural generalization of
differential algebras that permits the algebraic formulation of linear
boundary problems.

\begin{definition}
  \label{def:intdiffalg}
  We call $(\galg, \der, \cum)$ an \emph{integro-differential algebra}
  over $K$ if $\galg$ is a commutative $K$-algebra with $K$-linear
  operations $\der$ and $\cum$ such that the three axioms
  \begin{gather}
    \label{eq:section-axiom}
    (\cum f)' = f,\\
    \label{eq:leibniz-axiom}
    (fg)' = f'g+ fg',\\
    \label{eq:diff-baxter-axiom}
    (\cum f')(\cum g') + \cum (fg)' = (\cum f')g + f (\cum g')
  \end{gather}
  are satisfied, where $\dots'$ is the usual shorthand notation for $\der$.
\end{definition}

This definition differs from Definition~4
in~\cite{RosenkranzRegensburger2008a} as it
uses~\eqref{eq:diff-baxter-axiom} as the \emph{differential Baxter
  axiom} rather than the variant
\begin{equation}
  \label{eq:old-diff-baxter-axiom}
  \cum fg = f \cum g - \cum f' \cum g
\end{equation}
used earlier. In fact, one checks immediately that for
commutative~$\galg$ both axioms are equivalent (assuming the other
axioms). The advantage of~\eqref{eq:diff-baxter-axiom} is that it is
more symmetric and that it also sufficient for the noncommutative
case, where~\eqref{eq:old-diff-baxter-axiom} must be supplemented by
the corresponding dual axiom. Nevertheless, \emph{we will from now on
  tacitly assume that all integro-differential algebras are
  commutative, as in~\cite{RosenkranzRegensburger2008a}.}

Incidentally we note the following \emph{convention on precedence}:
The scope of an integral sign $\cum$ covers all factors to its right,
unless otherwise specified. Therefore the above term~$\cum f' \cum g$
is to be parsed as~$\cum (f' \cum g)$. This helps avoiding the
proliferation of parentheses in nested integrals.

The differential Baxter axiom is in general stronger than the
following \emph{pure Baxter axiom}
\begin{equation}
  \label{eq:pure-baxter-axiom}
  (\cum f)(\cum g) = \cum f \cum g + \cum g \cum f,
\end{equation}
which is the defining axiom for the so-called Rota-Baxter
algebras~\cite{Guo2002,Baxter1960,Rota1969}. (In fact, one may relax
the conditions on~$\galg$ to include noncommutative algebras over
arbitrary commutative rings with one. Moreover, one may add a
so-called weight term for incorporating the discrete setting
where~$\cum$ is, for example, the operator of partial summation. In
that case, one must of course also adapt~\eqref{eq:leibniz-axiom} to
account for the weight. Remarkably, the differential Baxter
axiom~\eqref{eq:diff-baxter-axiom} is the same with or without
weight.)

Every integro-differential algebra comes with a multiplicative
projector, namely the \emph{evaluation}~$\evl = 1 - \cum \der$. In
fact, the strong Baxter axiom is equivalent to the weak one combined
with the multiplicativity of~$\evl$. The presence of the
character~$\evl$ leads to the direct sum decomposition
\begin{equation}
  \label{eq:direct-sum}
  \galg = \const \dirs \init
\end{equation}
where~$\const = \Img{\evl} = \Ker{\der}$ is the usual \emph{ring of
  constants} while~$\init = \Ker{\evl} = \Img{\cum}$ is called the
\emph{ideal of initialized functions} (one checks immediately that the
multiplicativity of~$\evl$ is equivalent to~$\init$ being an
ideal). So~$\evl$ is the projector onto~$\const$
along~$\init$. Moreover, both~$\der$ and~$\cum$ are~$\const$-linear
rather than just~$K$-linear (for~$\der$ this is of course trivial, but
for~$\cum$ this is again equivalent to the multiplicativity
of~$\evl$).

In this paper, we want to restrict ourselves to boundary problems for
LODEs. For reflecting this property in the (integro) differential
structure, in~\cite{RosenkranzRegensburger2008a} we have called a
differential algebra~$(\galg, \der)$ ordinary if~$\Ker{\der} =
K$. \emph{We will henceforth assume that all integro-differential
  algebras are ordinary in this sense.}

Restricting ourselves to ordinary integro-differential algebras has a
number of pleasant implications. First of all, the evaluation is now a
multiplicative linear operator~$\evl\colon \galg \rightarrow \galg$,
meaning a \emph{character} (in the sense ``multiplicative linear
functional on an algebra''). Another consequencence, which will be of
some importance later, is that the polynomials behave as usual. For
any integro-differential algebra~$\galg$, we set~$x = \cum 1 \in
\galg$. Now the pure Baxer axiom~\eqref{eq:pure-baxter-axiom} ensures
that~$x^2 = 2 \cum \cum 1 \in \galg$, and so on. It turns out that the
elements of~$K[x]$ are ``really'' polynomials, and they satisfy the
usual differential equations (see before Eq.~(6) and also Eq.~(8)
in~\cite{RosenkranzRegensburger2008a} for more details).

\begin{proposition}
  \label{prop:pol-ring}
  Let~$\galg$ be an integro-differential algebra over~$K$, and
  let~$K[x]$ be the subalgebra generated by~$x = \cum 1$.  Then~$K[x]$
  is isomorphic to the univariate polynomial ring over~$K$. If~$\galg$
  is an ordinary integro-differential algebra, we have moreover
  \begin{equation}
    \label{eq:pol-fund-sys}
    \Ker{\der^n} = [1, x, \dots, x^{n-1}],
  \end{equation}
  and~$K[x] \le \galg$ is an integro-differential subalgebra.
\end{proposition}

\begin{example}
  \label{ex:standard}
  The \emph{standard example} of an (ordinary) integro-differential
  algebra is~$\galg = C^\infty(\R)$ with~$\der u = u'$ in the usual
  sense and $\cum u = \cum_0^x u(\xi) \, d\xi$. This
  integro-differential algebra contains many important
  integro-differential subalgebras, for example the analytic
  functions~$C^\omega(\R)$, the exponential polynomials, and of course
  the polynomial ring~$\R[x]$. An important integro-differential
  subalgebra of~$C^\omega(\R)$ is formed by the holonomic power
  series~\cite{Chyzak1994,SalvyZimmerman1994}.
\end{example}

Given a differential algebra~$(\galg, \der)$, one may form the ring of
differential operators~$\gdiffop$. It is therefore natural to expect a
similar ring of \emph{integro-differential operators}~$\intdiffop$,
which contains differential operators (needed for specifying
differential equations) as well as integral operators (needed for
specifying Green's operators). The evaluation~$\evl$ can be used for
writing a condition like~$2u'(0) - 3 (0) = 0$ in the form~$(2 \evl
\der - 3 \evl) \, u = 0$. But for boundary problems one usually needs
more than one evaluation (see Example~\ref{ex:second-order} below).

For algebraizing general boundary conditions, we start from a
\emph{character set}~$\Phi$, meaning each~$\phi \in \Phi$ is a
multiplicative linear functional just as~$\evl$ is. In fact, we will
always assume~$\evl \in \Phi$. Based on the characters of~$\Phi$, we
can from arbitrary \emph{Stieltjes
  conditions}~\cite[Def.~14]{RosenkranzRegensburger2008a}; their
normal form is
\begin{equation}
  \label{eq:stieltjes-cond}
  \sum_{\phi \in \Phi} \left( \sum_{i \in \N}
    a_{\phi,i} \, \phi \der^i + \phi \cum f_{\phi} \right)
\end{equation}
with $a_{\phi,i} \in K$ and $f_{\phi} \in \galg$ almost all zero. The
summands with~$\cum$ make up the so-called global part, those without
the local part. The \emph{order} of a boundary condition is the
largest~$i$ such that~$a_i \ne 0$ in the normal
form~\eqref{eq:stieltjes-cond}. In the standard setting of
Example~\ref{ex:standard}, a typical Stieltjes condition like
\begin{equation*}
  u''(0) - 3 u(-1) + 7 u(1) + \int_0^1 \xi^2 u(\xi) \, d\xi -
  \int_{-1}^1 e^\xi u(\xi) \, d\xi = 0
\end{equation*}
is encoded as~$\beta(u) = 0$ with~$\beta = \evl_0 \der^2 - 3 \evl_{-1}
+ 7 \evl_1 + \evl_1 \cum x^2 - \evl_1 \cum e^x + \evl_{-1} \cum e^x$,
where we have written~$\evl_a$ for the standard character~$f \mapsto
f(a)$. Henceforth we shall always consider Example~\ref{ex:standard}
with the character set~$\Phi = \{ \evl_a \mid a \in \R \}$.

For each character set~$\Phi$ one may then build up the corresponding
\emph{operator ring}~$\intdiffop[\Phi]$. The details are given
in~\cite[\S3]{RosenkranzRegensburger2008a}; at this point it suffices
to note that the operator ring is given by the quotient of a free
algebra modulo a certain ideal of relations like the Leibniz rule, and
also~\eqref{eq:diff-baxter-axiom} as well
as~\eqref{eq:pure-baxter-axiom}. The resulting operator ring
decomposes nicely into a direct sum;
see~\cite[Prop.~17]{RosenkranzRegensburger2008a} for a proof.

\begin{proposition}
  \label{prop:intdiffop-direct-sum}
  For any integro-differential algebra~$(\galg, \der, \cum)$ and
  character set~$\Phi$, we have the decomposition~$\intdiffop[\Phi] =
  \gdiffop \dirs \gintop \dirs (\Phi)$.
\end{proposition}

Here~$(\Phi)$ is the two-sided ideal generated by~$\Phi$ in the
operator ring~$R = \intdiffop[\Phi]$. We refer to its elements as
\emph{boundary operators} since they may also be described in terms of
Stieltjes conditions: One checks first that the collection of
Stieltjes conditions can be characterized as the right
ideal~$\mathopen|\Phi) \equiv \Phi R$. It turns out that~$(\Phi)$ is
then the left~$\galg$-module generated by~$\mathopen|\Phi)$, so every
boundary operator can be written as~$f_1 \beta_1 + \cdots + f_n
\beta_n$ for some functions~$f_1, \dots, f_n \in \galg$ and Stieljes
conditions~$\beta_1, \dots, \beta_n \in (\Phi)$. In particular, every
Stieltjes condition is also a boundary operator.

The decomposition~$\intdiffop[\Phi] = \gdiffop \dirs \gintop \dirs
(\Phi)$ reflects the basic needs for an \emph{algebraic formulation of
  boundary problems}: We need~$\gdiffop$ for the differential
equation, $(\Phi)$ for the boundary conditions, and~$\gintop$ for the
solution (Green's operators). Let us see how these ingredients look
like for the simplest possible boundary problem in the standard
setting of Example~\ref{ex:standard}.

\begin{example}
  \label{ex:second-order}
  The boundary problem
  \begin{equation*}
    \bvp{u'' = f}{u(0) = u(1) = 0}
  \end{equation*}
  can be encoded in the standard setting~$\galg = C^\infty(\R)$ by the
  differential operator~$T = \der^2 \in \gdiffop$ and the boundary
  conditions~$\evl_0 = \evl, \evl_1 \in (\Phi)$. Here we can
  choose~$\Phi = \{ \evl_0, \evl_1 \}$ or any character set containing
  that. The Green's operator can be written as~$G = Ax + xB - xAx -
  xBx \in \gintop$, where we use the standard abbreviations~$A = \cum$
  and~$-B = (1-\evl_1) \cum$. Note that~$A$ is the integral from~$0$
  to~$x$, and~$B$ the integral from~$x$ to~$1$. In the $L^2$ setting,
  $A$ and~$B$ are adjoint operators.
\end{example}

The algebraic description in terms of the two functionals~$\evl_0$
and~$\evl_1$ contains some arbitrariness since we can clearly form any
linear combination of the two functionals without changing the
solution operator~$G$. In general, a
\emph{boundary space} is a finite-dimensional subspace of~$\galg^*$
generated by Stieltjes conditions. So the boundary space of
Example~\ref{ex:second-order} is~$[\evl_0, \evl_1]$. The lattice of
all boundary spaces will be denoted by~$\bspclat$.

We will usually restrict the coefficient functions of the differential
operators to a differential subalgebra~$\calg \le \galg$ so as to
ensure solutions in the ambient algebra~$\galg$. This means we require
that~$\calg$ be \emph{saturated} for~$\galg$ in the sense
of~\cite[Def.~18]{RosenkranzRegensburger2008a}. \emph{Whenever an
  integro-differential algebra~$(\galg, \der, \cum)$ is specified, it
  is understood that a saturated coefficient algebra~$\calg$ is set
  aside.}

\begin{definition}
  \label{def:boundary-problem}
  A \emph{boundary problem} is a pair~$(T, \bspc)$ consisting of a
  monic differential operator~$T \in \calg[\der]$ and a boundary
  space~$\bspc \in K\Phi$.
\end{definition}

So the boundary problem in Example~\ref{ex:second-order} is given
by~$(\der^2, [\evl_0, \evl_1])$. Conversely, we can think of a
boundary problem~$(T, \bspc)$ as finding a solution~$u \in \galg$ such
that
\begin{equation*}
  \bvp{Tu=f}{\beta(u) = 0 \quad (\beta \in \bspc)}
\end{equation*}
for an arbitrary forcing function~$f \in \galg$. Of course it suffices
that~$\beta$ ranges over any $K$-basis of~$\bspc$; this is how the
boundary conditions are normally given in the first place. The
\emph{traditional formulation} in terms of~$u$ and~$f$ is more
intuitive, but it conceals the fact that we are really working in the
operator ring~$\intdiffop[\Phi]$.

Clearly every differential operator~$T$ has a certain order written
as~$\ord{T}$, which as in Example~\ref{ex:second-order} usually
coincides with the number of given boundary conditions (the dimension
of~$\bspc$). In fact, this a necessary---but in general not
sufficient---condition for the boundary problem to be \emph{regular}
in the sense that there is a unique solution~$u \in \galg$ for every
given forcing function~$f \in \galg$. In turns out that this is
equivalent to the following formulation in terms of the relevant
spaces.

\begin{definition}
  \label{def:regular-bp}
  A boundary problem~$(T, \bspc)$ is called \emph{regular} if~$\Ker{T}
  \dirs \orth{\bspc} = \galg$, and \emph{singular} otherwise.
\end{definition}

Here we have written~$\orth{\bspc}$ for the \emph{orthogonal}, meaning
the space of all~$f \in \galg$ such that~$\beta(f) = 0$ for all~$\beta
\in \bspc$. In other words, $\orth{\bspc}$ is the space of admissible
functions: those that satisfy the given boundary conditions. The
orthogonal, together with the analogous notion~$\orth{\fspc}$ for
subspaces~$\fspc \le \galg$, establishes a Galois connection
between~$\galg$ and its dual~$\galg^*$. This is similar to the
situation in algebraic geometry with its Galois connection between
affine varieties and radical
ideals. See~\cite[\S5]{RosenkranzRegensburger2008a} for more details.

The above definition of regularity is not suitable for algorithmic
purposes. However, it turns out to be equivalent to the following
explicit \emph{regularity
  test}~\cite[Prop.~6.1]{RegensburgerRosenkranz2009}, which is
well-known in the special case of local boundary
conditions~\cite[p.~184]{Kamke1967}.

\begin{lemma}
  \label{lem:regularity-test}
  Let~$(T, \bspc)$ be a boundary problem over~$\galg$ and choose
  bases~$\beta_1, \dots, \beta_n$ for~$\bspc$ and~$u_1, \dots, u_n$
  for~$\Ker{T}$. Then~$(T, \bspc)$ is regular iff~$\ord{T} =
  \dim{\bspc}$ and the \emph{evaluation matrix} $\beta(u) =
  [\beta_i(u_j)] \in K^{n \times n}$ is regular.
\end{lemma}

Regular boundary problems are exactly those that have a \emph{Green's
  operator} in the classical sense
(see~\cite{KorporalRegensburgerRosenkranz2011} on generalized Green's
operators for singular boundary problems). In our algebraic setting,
Green's operators can be defined within the operator
ring~$\intdiffop[\Phi]$.

\begin{definition}
  \label{def:greens-operator}
  Let~$(T, \bspc)$ be a regular boundary problem over an
  integro-differential algebra~$\galg$. Then~$G \in \intdiffop[\Phi]$
  is called its \emph{Green's operator} if~$TG = 1$ and~$\Img{G} =
  \orth{\bspc}$.
\end{definition}

The two conditions for Green's operators are just that $u = Gf$
satisfies the differential equation~$Tu = f$ as well as all boundary
conditions~$\beta(u) = 0$ for~$\beta \in \bspc$. Since we have assumed
a coefficient algebra~$\calg$ saturated for~$\galg$, one can
prove~\cite[Thm.~26]{RosenkranzRegensburger2008a} that \emph{every
  regular boundary problem} has a Green's operator
in~$\intdiffop[\Phi]$. In a leap of faith, we write~$G = (T,
\bspc)^{-1}$ just as
in~\cite[p.~533]{RosenkranzRegensburger2008a}. The relation to actual
inverses will become clear in
Proposition~\ref{prop:greensop-localization}.

Apart from the localization, there is another good reason for the
notation~$G = (T, \bspc)^{-1}$. It turns out that boundary problems
can be multiplied in such a way that
\begin{equation}
  \label{eq:anti-isomorphism}
  \Big( (T_1, \bspc_1) (T_2, \bspc_2) \Big)^{-1} = 
  (T_2, \bspc_2)^{-1} \, (T_1, \bspc_1)^{-1}
\end{equation}
is satisfied. For that purpose, we have defined the multiplication of
boundary problems in the fashion of a semi-direct product by
\begin{equation}
  \label{eq:bp-product}
  (T_1, \bspc_1) (T_2, \bspc_2) = (T_1 T_2, \bspc_1 T_2 +
  \bspc_2).
\end{equation}
It is easy to check that one obtains a \emph{monoid}~$\allprob$ in
this way, with~$(1, O)$ as the neutral element. Furthermore, it
turns out that the regular boundary problems form a
submonoid~$\regprob \subset \allprob$, and~\eqref{eq:anti-isomorphism}
means that~$(T, \bspc) \mapsto (T, \bspc)^{-1}$ is an anti-isomorphism
from~$\regprob$ to the multiplicative monoid of Green's operators. For
a proof of these facts, we refer
to~\cite[Prop.~27]{RosenkranzRegensburger2008a}.

The monoid algebra~$K \regprob$ will be called the \emph{ring of
  boundary problems} (we suppress the qualification ``regular'' since
we will not investigate the singular case in the frame of this
paper). Its elements will supply the numerators in the localization to
be constructed below (Section~\ref{sec:localization}).

At this point we should also note that a regular boundary problem~$(T,
\bspc)$ need not be \emph{well-posed} in the sense that its Green's
operator~$G$ is continuous in the standard setting of
Example~\ref{ex:standard}. For example, consider the regular boundary
problem~$(\der - 1, \evl_0 \der^2)$ or
\begin{equation*}
  \bvp{u'-u=f}{u''(0)=0}
\end{equation*}
in traditional notation. One checks immediately that~$G = e^x \cum
e^{-x} - e^x \evl_0 - e^x \evl_0 \der \in \intdiffop[\Phi]$ is its
Green's operator, meaning the general solution is given by
\begin{equation*}
  u(x) = \cum_0^x e^{x-\xi} f(\xi) \, d\xi - (f(0) + f'(0)) \, e^x.
\end{equation*}
Clearly, this Green's operator~$G$ is not continuous, at least not if
one endows~$C^\infty(\R)$ with the usual topology induced by~$C(\R)$.
Following Hadamard, a well-posed boundary
problem~\cite[p.~86]{Engl1997} must be regular (meaning the solution
$u$ exists and is unique for each given~$f$) as well as stable
(meaning the solution $u$ depends continuously on $f$). In the example
above, the source of the instability is clear---we have imposed a
second-order condition for a first-order differential equation. Going
back to the algebraic setting then, we call a boundary problem~$(T,
\bspc)$ well-posed if it is regular and~$\bspc$ can be generated by
Stieltjes conditions whose order is smaller than that of~$T$;
otherwise~$(T, \bspc)$ is called \emph{ill-posed}. As one easily
checks, the well-posed boundary problems form a submonoid
of~$\regprob$.

The Green's operator~$G$ of a regular boundary problem~$(T, \bspc)$ of
order~$n$ factors naturally as
\begin{equation}
  \label{eq:proj-times-fri}
  G = (1-P) \, \fri{T},
\end{equation}
where $\fri{T}$ is the so-called \emph{fundamental right inverse}
of~$T$, defined as the Green's operator of the initial value
problem~$(T, [\evl, \evl \der, \dots, \evl \der^{n-1}]$, and~$P$ is
the projector onto~$\Ker{T}$ along~$\orth{\bspc}$. The fundamental
right inverse~$\fri{T}$ exists since initial value problems are always
solvable~\cite[Prop.~22]{RosenkranzRegensburger2008a}. We may thus
think of~\eqref{eq:proj-times-fri} as a two step process:
Using~$\fri{T}$ one solves first the initial value problem (which is
usually much easier), then one incorporates the boundary conditions
via the projector~$P$.

The submonoid relation~$\regprob \subset \allprob$ means that
multiplying regular problems leads to regular problems again. It is
interesting to note that the converse is also true, provided the order
condition of Lemma~\ref{lem:regularity-test} holds.

\begin{lemma}
  \label{lem:reg-factors}
  Let $(T_1, \bspc_1), (T_2, \bspc_2)$ be boundary problems over an
  integro-differential algebra~$\galg$ with $\ord T_1 = \dim \bspc_1$
  and $\ord T_2 = \dim \bspc_2$. Then~$(T_1, \bspc_1)$ and~$(T_2,
  \bspc_2)$ are regular whenever $(T_1, \bspc_1) \cdot (T_2, \bspc_2)$
  is.
\end{lemma}
\begin{proof}
  Let us write~$m = \ord T_1 = \dim \bspc_1$ and $n = \ord T_2 = \dim
  \bspc_2$. Choose fundamental systems~$f_1, \dots, f_m \in \galg$
  for~$T_1$ and~$g_1, \dots, g_n \in \galg$
  for~$T_2$. Take~$K$-bases~$\beta_1, \dots, \beta_m$ of~$\bspc_1 \le
  \galg^*$ and~$\gamma_1, \dots, \gamma_n$ of~$\bspc_2 \le
  \galg^*$. Then~$\fri{T_2}\! f_1, \dots, \fri{T_2}\! f_m, g_1, \dots,
  g_n$ is a fundamental system for~$T_1T_2$, and a~$K$-basis
  of~$\bspc_1 T_2+\bspc_2$ is given by $\beta_1 T_2, \dots, \beta_m
  T_2$, $\gamma_1, \dots, \gamma_n$. The latter fact follows since the
  sum~$\bspc_1 T_2 + \bspc_2$ in~\eqref{eq:bp-product} is always
  direct~\cite[Prop.~3.2]{RegensburgerRosenkranz2009}. By
  Lemma~\ref{lem:regularity-test} the regularity of the boundary
  problem~$(T_1T_2, \bspc_1 T_2 + \bspc_2)$ means that its evaluation
  matrix
  \begin{equation*}
    \begin{pmatrix}
      (\beta T_2)(\fri{T_2}\! f) & (\beta T_2)(g)\\
      \gamma(\fri{T_2}\! f) & \gamma(g)
    \end{pmatrix}
    =
    \begin{pmatrix}
      \beta(f) & 0\\
      \gamma(\fri{T_2}\! f) & \gamma(g)
    \end{pmatrix}
  \end{equation*}
  is regular. But this is only possible if both diagonal
  blocks~$\beta(f)$ and~$\gamma(g)$ are regular, and these are just
  the evaluation matrices of~$(T_1, \bspc_1)$ and~$(T_2, \bspc_2)$.
\end{proof}

For practical applications, one is not so much interested in
multiplying boundary problems---thus increasing their order---as to
factor them into lower-order problems. Interestingly, factorization
will also be instrumental in the localization process
(Lemma~\ref{lem:regularization}). We repeat here as the main result
the so-called \emph{Factorization
  Theorem}~\cite[Thm.~32]{RosenkranzRegensburger2008a}.

\begin{theorem}
  \label{thm:fact-bp}
  Given a regular boundary problem~$(T, \bspc) \in \regprob$, every
  factorization~$T = T_1 T_2$ of the differential operator lifts to a
  factorization~$(T, \bspc) = (T_1, \bspc_1) \cdot (T_2, \bspc_2)$
  with~$(T_1, \bspc_1), (T_2, \bspc_2)$ regular and~$\bspc_2 \le
  \bspc$.
\end{theorem}

We note that the right factor may be chosen to be an arbitrary
subspace of~$\bspc$ as long as regularity holds for~$(T_2, \bspc_2)$;
in particular one can always choose initial conditions for~$\bspc_2$
to ensure regularity. In contrast, $\bspc_1$ is determined by the
choice of~$T_1$ and~$T_2$ alone. This is the content of the following
uniqueness result~\cite[Prop.~31]{RosenkranzRegensburger2008a}, called
the \emph{Division Lemma} for boundary problems (see below for the
notion of subproblem).

\begin{lemma}
  \label{lem:division}
  Given a regular boundary problem~$(T, \bspc) \in \regprob$ and any
  factorization $T = T_1 T_2$ of the differential operator, there is a
  unique boundary problem~$(T_1, \bspc_1)$ such that for any regular
  subproblem~$(T_2, \bspc_2) \le (T, \bspc)$ we have the lifted
  factorization~$(T_1, \bspc_1) \cdot (T_2, \bspc_2) = (T, \bspc)$.
\end{lemma}

The \emph{notion of subproblem} was defined (under the name ``right
factor'', which is avoided here due to the ambiguity in cases like
Example~\ref{ex:embed-single-cond}) for regular boundary problems
in~\cite[Def.~29]{RosenkranzRegensburger2008a}, but it can be extended
naturally to all boundary problems. Hence a subproblem of a boundary
problem~$(T, \bspc) \in \allprob$ is any~$(T_2, \bspc_2) \in \allprob$
with~$T_2$ a right divisor of~$T$ and $\bspc_2$ a subspace
of~$\bspc$. We denote this relation by~$(T_2, \bspc_2) \le (T,
\bspc)$.

We conclude this section with a few remarks on \emph{algorithmic
  issues}. While at the present stage of our work we do not aim at a
computational realization of the localization, we see this
nevertheless as a mid-term goal. In fact, most of the results in this
paper are algorithmic, with two exceptions: (1) The Ore condition is
needed in~$\calg[\der]$ for monic operators; this is not supported by
customary packages but appears to be adaptable (see the remarks after
Proposition~\ref{prop:left-ext}). (2) One would need an algorithm for
determining the monomials required in
Definition~\ref{def:umbral-characters}. Once these two gaps are
filled, one should be able to compute in the localization since all
the constructions reviewed in this section are themselves
algorithmic~\cite{RosenkranzRegensburger2008a}:
\begin{itemize}
\item The \emph{operator ring}~$\intdiffop[\Phi]$ is defined as a
  quotient by a certain ideal of relations. This ideal is described by
  a noncommutative Gr\"obner basis for an ideal with infinitely many
  generators (on another view: a noetherian and confluent rewrite
  system).
\item If the \emph{ground algebra}~$\galg$ and the \emph{character
    set}~$\Phi$ are computable, the normal forms of~$\intdiffop[\Phi]$
  are computable as well. Typically, $\galg$ is an algorithmic
  fragment of~$C^\infty(\R)$ like the exponential
  polynomials~\cite[p.~176]{Rosenkranz2005}, while~$\Phi$ consists of
  finitely many point evaluations~$f \mapsto f(a)$.
\item The \emph{Green's operator} of a regular boundary problem (and
  hence the solutions for arbitrary forcing functions up to
  quadratures) can be computed as long as one has a fundamental system
  for the underlying homogenous differential equation. For the latter,
  one can in principle rely on the vast body of results from
  differential Galois theory~\cite{PutSinger2003}.
\item The \emph{multiplication and factorization} of boundary problems
  can be computed as long as the same is true for the constituent
  differential operators. In this way one can often solve higher-order
  boundary problems by decomposing them into smaller factors such as
  in a recent application to actuarial
  mathematics~\cite{AlbrecherConstantinescuPirsicEtAl2009}. For the
  algorithmic theory of factoring linear differential operators, we
  refer to~\cite{Grigoriev1990,Schwarz1989,Tsarev1996}.
\end{itemize}
As detailed in the Introduction, the algorithms addressed above are
\emph{implemented} in \mma\ and \mpl\ packages.

\section{Umbral Boundary Conditions}
\label{sec:umbral-bc}

We have seen that for a given integro-differential algebra~$\galg$ and
character set~$\Phi$, the natural choice of boundary conditions is the
\emph{Stieltjes conditions}, defined
in~\cite{RosenkranzRegensburger2008a} as the right
ideal~$\mathopen|\Phi)$ of~$\intdiffop[\Phi]$ and characterized by
their normal forms~\eqref{eq:stieltjes-cond}. The choice
of~$\mathopen|\Phi)$ is vindicated by Theorem~\ref{thm:fact-bp}, which
asserts that Stieltjes conditions are sufficient for describing
arbitrary factor problems (and they are clearly also sufficient for
multiplying boundary problems). In fact, Stieltjes conditions with
nonzero global part appear in the left factor problem even if the
original boundary problem has only local conditions; see
e.g. Example~33 in~\cite{RosenkranzRegensburger2008a}.

For the purpose of localization we must isolate a subclass of
Stieltjes conditions. First of all, we must distinguish carefully
between the zero condition~$0 \in \mathopen|\Phi)$ and degenerate
boundary conditions that only \emph{act} as zero: We call a Stieltjes
condition~$\beta$ \emph{degenerate} if~$\phi(f) = 0$ for all~$f \in
\galg$. In the standard setting of Example~\ref{ex:standard}, there
are plenty of degenerate boundary conditions. If~$f$ is a bump function
supported, say, on an interval disjoint from~$[0, 1]$, the Stieltjes
condition~$\evl_1 \cum f = \cum_0^1 \, f$ is clearly degenerate. In
contrast, the integro-differential subalgebra of analytic
functions~$C^\omega(\R)$ does not have degenerate global conditions of
this form as we shall see below (Example~\ref{ex:complete}).

The subclass of Stieltjes conditions to be chosen must be so as to
ensure a sufficient supply of regular boundary problems. In
particular, we shall find ourselves in the situation of embedding a
singular boundary problem~$(T, \bspc)$ into a surrounding regular
problem (Lemma~\ref{lem:regularization}). Hence we must enlarge~$T \in
\diffop$ by a suitable differential operator~$\tilde{T} \in \diffop$
and the boundary space~$\bspc$ by \emph{new boundary
  conditions}~$\tilde{\beta}_1, \tilde{\beta}_2, \dots$ in such a
manner that the resulting evaluation matrix is regular
(Lemma~\ref{lem:regularity-test}).

The key for solving this problem is found in the fortunate fact every
integro-differential algebra~$\galg$ contains the polynomial ring
(Proposition~\ref{prop:pol-ring}), so the \emph{initial value problem}
\begin{equation*}
  (\der^n, [\evl, \evl \der, \dots, \evl \der^{n-1}])
\end{equation*}
is a natural choice for extending the given boundary problem~$(T,
\bspc)$. Then the monomials~$1, x/1!, \dots, x^{n-1}/(n-1)!$ are a
fundamental system by Proposition~\ref{prop:pol-ring}, and the
corresponding evaluation matrix is~$I_n$. But this is not enough for
embedding singular problems since the evaluation matrix of the
surrounding problem involves combining the \emph{given} boundary
conditions~$\beta \in \bspc$ with the monomials of the fundamental
system~\eqref{eq:pol-fund-sys}. Indeed, the crucial step in the proof
of Lemma~\ref{lem:regularization} needs that~$\beta(x^m) \ne 0$
for~\emph{some} monomial $x^m \in K[x]$. This is the motivation for
the following definition.

\begin{definition}
  \label{def:umbral-characters}
  A Stieltjes condition~$\beta \in \intdiffop[\Phi]$ is called
  \emph{umbral} if~$\beta(x^m) \ne 0$ for some monomial~$x^m \in
  K[x]$. Furthermore, we call~$\Phi$ an \emph{umbral character set} if
  every nondegenerate Stieltjes condition is umbral.
\end{definition}

The reason for our terminology is that there is an interesting link
between Stieljes conditions~$\beta \in \intdiffop[\Phi]$
with~$\beta(x^m) \ne 0$ and the \emph{umbral calculus}. Indeed, every
umbral condition defines a nontrivial shift-invariant operator
(Proposition~\ref{prop:umbral-cond}). This is clear for the local
parts~$\phi \der^n$ but needs some justification for global
terms~$\phi \cum f$ with arbitrary~$f \in \galg \supset K[x]$.  The
point is that one may apply the following special case of the
well-known ``antiderivative Leibniz rule''. (In the general setting of
analytic functions one writes~$\cum fg$ as an infinite series of
iterated integrals of~$f$ times iterated derivatives of~$g$. Here we
have~$g=x^n$ so that only finitely many derivatives are nonzero and
the series terminates.)

\begin{lemma}
  \label{lem:int-part-pol}
  In any integro-differential algebra~$(\galg, \der, \cum)$, we have
  the formula
  \begin{equation}
    \label{eq:int-part-pol}
    \cum fx^n = \sum_{k=0}^n (-1)^k \fafac{n}{k} \, x^{n-k} f^{(-k-1)}
  \end{equation}
  for all $f \in \galg$. Here $f^{-k} \; (k \ge 0)$ is defined by
  $f^{(0)} = f$ and $f^{(-k-1)} = \cum f^{(-k)}$.
\end{lemma}
\begin{proof}
  This is essentially the special case~$m=1$ of Equation~(14)
  in~\cite{RegensburgerRosenkranzMiddeke2009}, but the proof is short
  enough to present here. We use induction over~$n$. The base
  case~$n=0$ is clear, so assume~\eqref{eq:int-part-pol} for~$n\ge 0$;
  we prove it for~$n+1$. Since $x^{n+1} = (n+1) \, \cum x^n$ we have
  \begin{equation*}
    \cum f x^{n+1} = (n+1) \, \cum f \cum x^n = f^{(-1)} x^{n+1} -
    (n+1) \, \cum f^{(-1)} x^n
  \end{equation*}
  by the pure Baxter axiom~\eqref{eq:pure-baxter-axiom}. Using the
  induction hypothesis, the right summand becomes
  \begin{equation*}
    \sum_{k=1}^{n+1} (-1)^k \fafac{(n+1)}{k} \, x^{n+1-k} f^{(-k-1)} 
  \end{equation*}
  after an index transformation. Adding the extra term~$f^{(-1)}
  x^{n+1}$ means extending the summation to~$k=0, \dots, n+1$,
  which yields~\eqref{eq:int-part-pol} for~$n+1$.
\end{proof}

\begin{lemma}
  \label{lem:umbral-global-cond}
  Let~$\beta = \phi \cum f$ be a global condition
  in~$\intdiffop[\Phi]$. Then we have~$\beta = \phi \tilde{\beta}$ as
  a functional~$K[x] \rightarrow K$, where
  \begin{equation}
    \label{eq:umbral-exp}
    \tilde{\beta} = \sum_{k=0}^\infty b_k  \, \der^k : \quad
    K[x] \rightarrow K[x]
  \end{equation}
  is a shift-invariant operator with expansion coefficients~$b_k =
  (-1)^k \phi(f^{(-k-1)})$.
\end{lemma}
\begin{proof}
  By a well-known result of the umbral
  calculus~\cite[Thm.~2.1.7]{DiBucchianico1998}, incidentally also
  presented in~\cite[Prop.~92]{Rosenkranz1997}, the
  operator~$\tilde{\beta}$ is shift-invariant. For seeing that~$\beta
  = \phi \tilde{\beta}$, we apply Lemma~\ref{lem:int-part-pol} to
  obtain $\beta(x^n) = \sum_k b_k \, \phi(\fafac{n}{k} \, x^{n-k}) =
  \sum_k b_k \, \phi \der^k(x^n)$.
\end{proof}

The \emph{operator expansion}~\eqref{eq:umbral-exp} allows us to
associate a shift invariant operator with a Stieltjes condition of the
special form~$\beta = \phi \cum f$. But this association generalizes
immediately to arbitrary Stieltjes conditions since the local
conditions are unproblematic.

\begin{proposition}
  \label{prop:umbral-cond}
  Let~$\beta$ be a Stieltjes condition in~$\intdiffop[\Phi]$. Then
  there is an \emph{associated shift-invariant operator}
  \begin{equation}
    \label{eq:stieltjes-shift-inv}
    \tilde{\beta} = \sum_{k=0}^\infty b_k \der^k : \quad
    K[x] \rightarrow K[x]
  \end{equation}
  with coefficients~$b_k = \beta(x^k/k!)$ such that~$\beta = \evl \,
  \tilde{\beta}$. Clearly, the associated operator~$\tilde{\beta}$ is
  nonzero iff~$\beta$ is an umbral condition.
\end{proposition}
\begin{proof}
  Let~$\shift{\phi}\colon K[x] \rightarrow K[x]$ be the shift
  operator~$f(x) \mapsto f(x+\bar{\phi})$ with~$\bar{\phi} = \phi(x)
  \in K$. Using the normal form~\eqref{eq:stieltjes-cond} we can write
  the given boundary condition as
  \begin{equation*}
    \tilde{\beta} =  \sum_{\phi \in \Phi} \phi (T_\phi +
    \tilde{\beta}_\phi) = \evl \sum_{\phi \in \Phi} S_\phi (T_\phi +
    \tilde{\beta}_\phi),
  \end{equation*}
  where~$T_\phi = \sum_i a_{\phi,i} \der^i$ is clearly shift-invariant
  while~$\tilde{\beta}_\phi$ is the shift-invariant operator
  corresponding to~$\beta_\phi = \phi \cum f_\phi$ according to
  Lemma~\ref{lem:umbral-global-cond}. Then the terms~$S_\phi (T_\phi +
  \tilde{\beta}_\phi)$ in the second sum are shift-invariant, hence we
  have~$\beta = \evl \tilde{\beta}$ with the shift-invariant operator
  \begin{equation*}
    \tilde{\beta} = \sum_{\phi \in \Phi} S_\phi (T_\phi +
    \tilde{\beta}_\phi),
  \end{equation*}
  and the formula~$b_k = \beta(x^k/k!) = \evl \tilde{\beta}(x^k/k!)$
  for its expansion coefficients follows from the general result
  referred to in the proof of Lemma~\ref{lem:umbral-global-cond}.
\end{proof}

Umbral boundary conditions appear to be \emph{abundant}, at least in
the cases most crucial to us, especially in the~$C^\infty$ setting
(containing many important integro-differential subalgebras---notably
the analytic and holonomic functions as well as the exponential
polynomials).

\begin{proposition}
  \label{prop:cinf-weierstrass}
  In the standard setting of Example~\ref{ex:standard}, the point
  evaluations form an umbral character set.
\end{proposition}
\begin{proof}
  Consider an arbitrary nondegenerate Stieltjes condition
  \begin{equation*}
    \beta\colon \galg \rightarrow \R, \quad
    \beta(u) = \sum_{\phi \in \Phi} \sum_{i=0}^k a_{\phi,i} \,
    u^{(i)}(\phi) + \sum_{\phi \in \Phi}
    \int_0^\phi f_\phi(\xi) \, u(\xi) \, d\xi
  \end{equation*}
  so that there is a function~$u \in \galg$ with~$\beta(u) \ne
  0$. Let~$R$ be the maximum of $\abs{\phi}$ for all~$\phi$ with
  $a_{\phi,i} \ne 0$ or $f_\phi \ne 0$. We consider now the Banach
  space~$C^k(K)$ on the compact interval~$K = [-R, R]$. Its
  norm~$\norm{\cdot}_k$ is given by
  \begin{equation*}
    \norm{f}_k = \sum_{i=0}^k \norm{u^{(i)}}_\infty
  \end{equation*}
  for all~$f \in C^k(K)$. There is a little known generalization of
  the Weierstrass approximation theorem due to
  Nachbin~\cite{Nachbin1949}, which asserts (in a simple special case)
  that~$\R[x]$ is dense in~$C^k(K)$. Hence we may choose a polynomial
  sequence~$p_n$ that converges to~$u \in \galg \le C^k(K)$ in the
  $C^k$ topology.

  One checks immediately that~$\beta\colon C^k(K) \rightarrow \R$ is a
  continous functional with respect to the $C^k$~norm. In detail, one
  has~$\abs{\beta(u)} \le C \norm{u}_k$ with operator-norm bound
  \begin{equation*}
    C = \sum_{\phi \in \Phi} \sum_{i=0}^k a_{\phi,i} + \sum_{\phi \in
      \Phi} \abs{\phi} \, \norm{f_\phi}_\infty.
  \end{equation*}
  Therefore we have~$\beta(p_n) \rightarrow \beta(u) \ne
  0$. This is impossible if~$\beta$ vanishes on all of~$\R[x]$. Hence
  there is some~$p \in \R[x]$ with~$\beta(p) \ne 0$. Clearly there
  exists a smallest monomial~$x^m$ in~$p$ with~$\beta(x^m) \ne 0$.
\end{proof}

Not every character set is umbral, though. It is natural to introduce
the following \emph{necessary conditions} for a character set~$\Phi$
to be umbral (we write~$\bar{\phi}$ for the canonical value $\phi(x)
\in K$ of a character $\phi \in \Phi$):
\begin{enumerate}
\item The character set~$\Phi$ must clearly be \emph{separative} in
  the sense that~$\bar{\phi} = \bar{\chi}$ implies~$\phi =
  \chi$. Otherwise~$\beta = \phi - \chi$ is a nonzero Stieltjes
  condition with~$\beta(K[x]) = 0$.
\item Every character $\phi \in \Phi$ must be~\emph{complete} in the
  sense that every global condition of the form~$\beta = \phi \cum f$
  is umbral whenever it is nondegenerate. This may also be expressed
  as~$f \perp_\phi K[x] \Rightarrow f = 0$. Here orthogonality refers
  to the nondegenerate bilinear form~$\inner{f}{g} = \phi \cum fg$. If
  this bilinear form is positive definite with~$K = \R$ or~$K = \C$,
  we have a pre-Hilbert space~$(\galg, \inner{\cdot}{\cdot}_\phi)$.
  Following the terminology of~\cite[V.24]{Bourbaki1987}, completeness
  of~$\phi$ then means that any $\phi$-orthonormal basis of~$K[x]$ is
  complete in~$(\galg, \inner{\cdot}{\cdot}_\phi)$.
\end{enumerate}
These conditions are most likely not sufficient. At the moment we
cannot give a counterexample for corroborating this claim. But it is
intuitively clear that the completeness properties associated with
each~$\phi$ separately cannot prevent linear dependencies between the
actions of distinct global conditions~$\phi \cum f$. It remains an
interesting task to formulate stronger conditions on~$\Phi$ that
ensure an umbral character set in a natural way.

\begin{example}
  \label{ex:separative}
  As an example of a \emph{nonseparative character set}, consider the
  exponential polynomials~$\exppoly$ with the nonstandard
  character~$\phi$ defined by~$\phi(e^x) = 1$ and~$\phi(x^n) = 1$ for
  all~$n$, effectively mixing evaluation at~$1 \in K$ for monomials
  with evaluation at~$0 \in K$ for the exponential. Choosing~$K = \R$
  for simplicity, let~$\evl$ be the honest evaluation~$f \mapsto
  f(1)$. Then clearly~$\phi$ and~$\evl$ coincide on~$K[x]$ while they
  are in fact distinct characters on~$\exppoly$ since~$\phi(e^x) = 1
  \ne e = \evl(e^x)$.
\end{example}

This shows that separativity is necessary but not
sufficient. Nevertheless, it ensures that all \emph{local boundary
  conditions} are indeed umbral.

\begin{proposition}
  \label{prop:local-bc-umbral}
  Let~$\Phi$ be a separative character set for an integro-differential
  algebra~$(\galg, \der, \cum)$. Then every local boundary condition
  is umbral.
\end{proposition}
\begin{proof}
  Using~\eqref{eq:stieltjes-cond}, every local boundary condition can
  be written in the form
  \begin{equation*}
    \beta = \sum_{i=1}^r \sum_{j=1}^s a_{ij} \, \phi_i \der^{j-1}.
  \end{equation*}
  Assuming~$\beta(K[x]) = 0$, we show that all coefficients~$a_{ij}$
  are zero. Collecting them in the vector
  \begin{equation*}
    a = (a_{1,1}, a_{1,2}, \ldots, a_{1,s}; \; \ldots; \; a_{r,1},
    a_{r,2}, \ldots, a_{r,s})^T \in K^n,
  \end{equation*}
  we obtain~$n = rs$ linear homogeneous equations for the
  unkowns~$a_{ij}$ by applying~$\beta$ to the monomials~$1, x, x^2/2!,
  \dots, x^{n-1}/(n-1)!$. The matrix of the corresponding system~$Ma =
  0$ is given by~$M = (M_{ns}(\bar{\phi}_1) \cdots
  M_{ns}(\bar{\phi}_r)) \in K^{n \times n}$ with blocks
  \begin{equation}
    \renewcommand{\arraystretch}{1.4}
    \label{eq:block-binom}
    M_{ns}(x) \equiv
    \begin{pmatrix}
      1 &&&&&\\
      x & 1\\
      \tfrac{x^2}{2} & x & 1\\
      \;\vdots && \ddots & \ddots\\
      \tfrac{x^{s-1}}{(s-1)!} &
      \tfrac{x^{s-2}}{(s-2)!} & \cdots & x & 1\\
      \tfrac{x^{s}}{s!} &
      \tfrac{x^{s-1}}{(s-1)!} & \cdots &
      \tfrac{x^2}{2} & x\\
      \;\vdots & \;\vdots & \ddots & \;\vdots & \;\vdots\\
      \tfrac{x^{n-1}}{(n-1)!} &
      \tfrac{x^{n-2}}{(n-2)!} & \cdots &
      \tfrac{x^{n-s+1}}{(n-s+1)!} &
      \tfrac{x^{n-s}}{(n-s)!}
    \end{pmatrix}
    \in K[x]^{n \times s} \; .
  \end{equation}
  Since~$\Phi$ is separative, the values~$\bar{\phi}_1, \dots,
  \bar{\phi}_r$ are mutually distinct. Then we may apply the
  subsequent Lemma~\ref{lem:det} to obtain~$\det{M} \ne 0$, hence~$Ma
  = 0$ has the unique solution~$a=0$.
\end{proof}

\begin{lemma}
  \label{lem:det}
  The determinant of~$M(x) = (M_{ns}(x_1), \dots, M_{ns}(x_r)) \in
  K[x_1, \dots, x_r]^{n \times n}$, with $n = rs$ and
  blocks~\eqref{eq:block-binom}, is given by
  \begin{equation}
    \label{eq:vandermonde-determinant}
    \det{M(x)} = V(r)^{s^2} \superfac{s-1}^r / \superfac{n-1},
  \end{equation}
  where~$V(r) = \displaystyle\prod_{1 \le i < j \le r} (x_j-x_i)$ is
  the~$r \times r$ Vandermonde determinant and~$\superfac{i} = 1! 2!
  \cdots i!$ denotes the superfactorial.
\end{lemma}
\begin{proof}
  The determinant is a special case of \cite[Thm.\
  1.1]{FloweHarris1993} or \cite[Thm.\ 20]{Krattenthaler1999}, apart
  from the linear factor $\frac{1}{(k-1)!}$ in $k$-th row of
  $M(x)$. Thus we have
  \begin{equation*}
    \det{M(x)} = V(r)^{s^2} \times \bigg( \prod_{i=1}^{r}
    \prod_{j=1}^{s-1} j! \bigg) \bigg( \prod_{k=0}^{n-1} \frac{1}{k!}
    \bigg)
    = V(r)^{s^2} \superfac{s-1}^r / \superfac{n-1},
  \end{equation*}
  which is~\eqref{eq:vandermonde-determinant}.
\end{proof}

Let us now turn to \emph{completeness}. In the analysis setting,
every~$\phi$ is complete since we know already the stronger result
that every Stieltjes condition is umbral
(Proposition~\ref{prop:cinf-weierstrass}). It is nevertheless
instructive to have a closer look at this case.

\begin{example}
  \label{ex:complete}
  Consider the point evaluation~$\phi = \evl_a$ in the \emph{standard
    setting} of Example~\ref{ex:standard}. Here we get an inner
  product
  \begin{equation*}
    \inner{f}{g}_a = \beta(g) = \cum_0^a f(\xi) g(\xi) \, d\xi,
  \end{equation*}
  which can be extended to all continuous functions~$f,g$ on~$[0,a]$
  and so gives rise to the pre-Hilbert Hausdorff space~$(C[0,a],
  \inner{\cdot}{\cdot}_\phi)$. It is well-known that~$(x^m)_{m \in
    \N}$ is a complete sequence in this
  space~\cite[V.24]{Bourbaki1987}. We may assume~$a=1$ by a scale
  transformation. By the usual Gram-Schmidt process, the
  monomials~$x^m$ can be transformed to an orthonormal basis~$(e_n)$
  of~$C[0,1]$, which consists in this case of the Legendre polynomials
  \begin{equation*}
    e_n = \frac{\sqrt{n+1/2}}{2^n n!} \, \frac{d^n}{dx^n} \, (x^2 -
    1)^n.
  \end{equation*}
  In every pre-Hilbert Hausdorff space~$E$, an orthonormal
  basis~$(e_m)$ has the following well-known
  property~\cite[Prop.~V.2.5]{Bourbaki1987}: If~$f \in E$ is nonzero,
  we have~$\inner{e_m}{f} \ne 0$ for some~$m \in \N$. Applying this
  to~$E = C[0,1] \supset C^\infty[0,1]$ and fixing any~$f \in
  C^\infty[0,1] \subset C^\infty(\R)$, we see that~$\phi$ is complete
  in the sense defined above. For if~$\inner{x^m}{f}_\phi = 0$ for
  all~$m \in \N$, then~$f$ vanishes on~$[0,1]$ and hence~$\phi \cum f$
  is degenerate.

  The above argument shows in fact that every global condition~$\phi
  \cum f$ over~$C^\infty(\R)$ is umbral or degenerate. Working
  over~$C^\omega(\R)$, we can exclude the degenerate case. Indeed, the
  identity theorem of complex analysis ensures that~$f=0$ whenever~$f$
  vanishes on~$[0,1]$. This is in stark contrast to the~$C^\infty$
  case, which has plenty of degenerate global conditions~$\phi \cum f$
  as we have observed at the beginning of Section~\ref{sec:umbral-bc}.
\end{example}

It becomes clear from the above example that completeness is really an
\emph{analytic property} of some sort. It is therefore not surprising
that one can construct ``purely algebraic'' examples lacking this
property.

\begin{example}
  \label{ex:incomplete}
  To give an example of an \emph{incomplete character} on~$\galg =
  \exppoly$ with~$K = \R$, we proceed similar to
  Example~\ref{ex:separative}. Define~$\phi(x^n) = 0$ for~$n>0$
  and~$\phi(e^x) = e$, $\phi(1) = 1$. Then clearly~$\beta = \phi \cum
  1$ is nondegenerate since~$\beta(e^x) = \phi \cum e^x = \phi(e^x-1)
  = e-1 \ne 0$. But we have~$\beta(x^m) = \phi \cum x^m =
  \phi(\smash{\tfrac{x^{m+1}}{m+1}}) = 0$ for all~$m \in \N$.
\end{example}

As noted above, separativity and completeness are most likely not
strong enough to ensure an umbral character set. Note that umbrality
of~$\Phi$ implies that two Stieltjes conditions in~$\intdiffop[\Phi]$
are \emph{linearly independent} on~$K[x]$ whenever they are linearly
independent on~$\galg$. If one of the two conditions is local and the
other global, it is reasonable to expect this property to follow from
completeness. This expections is fulfilled.

\begin{lemma}
  \label{lem:global-not-local}
  Let~$(\galg, \der, \cum)$ be an integro-differential algebra with a
  complete character~$\phi$. Then a nondegenerate global
  condition~$\phi \cum f$ never coincides on~$K[x]$ with any local
  condition based on~$\phi$.
\end{lemma}
\begin{proof}
  Assume~$\beta = \phi \cum f$ coincides on~$K[x]$ with $a_0 \phi +
  a_1 \phi \der + \cdots + a_s \phi \der^s$, where~$s$ is chosen
  minimal. Then the umbral expansion~\eqref{eq:umbral-exp} breaks off
  at~$k=s$, and we have~$\phi(f^{(-k-1)}) = 0$
  for~$k>s$. For~$\tilde{f} = f^{(-s-1)}$, we
  get~$\phi(\tilde{f}^{(-k-1)}) = 0$ for all~$k \in \N$, so the
  condition~$\phi \cum \tilde{f}$ is degenerate since~$\phi$ is
  complete. For showing that this cannot happen, it suffices to prove
  that all~$\phi \cum f^{(-k)}$ are nondegenerate whenever~$\phi \cum
  f$ is. We use induction on~$k \in \N$. The base case~$k=0$ is
  trivial, so assume the claim for fixed~$k$. By the induction
  hypothesis we can choose~$g \in \galg$ with~$\phi \cum f^{(-k)} g =
  1$. Using~\eqref{eq:old-diff-baxter-axiom} with~$g$ in place of~$f$
  and~$f^{(-k)}$ in place of~$g$, we have
  \begin{equation*}
    \cum f^{(-k)} g = f^{(-k-1)} g - \cum f^{(-k-1)} g',
  \end{equation*}
  so our choice of~$g$ entails~$1 = \phi(f^{(-k-1)} g) - \phi \cum
  f^{(-k-1)} g'$. If the first summand on the right-hand side is
  different from~$1$, then $g'$ witnesses to~$\phi \cum f^{(-k-1)}$
  being nondegenerate, and the induction is
  complete. Otherwise~$\phi(f^{(-k-1)})$ must be nonzero, and we may
  use the special case~$(\cum h)^2 = 2 \cum h \cum h$ of the pure
  Baxter axiom~\eqref{eq:pure-baxter-axiom} to derive
  \begin{equation*}
    (\phi \cum f^{(-k-1)})(f^{(-k)}) = \phi \cum f^{(-k)} \cum f^{(-k)}
    = \tfrac{1}{2} \, \phi (\cum f^{(-k)})^2 = \tfrac{1}{2} \,
    \phi(f^{(-k-1)})^2 \ne 0,
  \end{equation*}
  and again the induction is complete.
\end{proof}

\section{The Ring of Methorious Operators}
\label{sec:localization}

Let us start by reviewing the general setting for localization in a
noncommutative unital ring~$R$, following~\cite[\S~4.10]{Lam1999}. As
in the commutative case, the denominator set~$S$ (the elements that
should become invertible) must clearly form a \emph{multiplicative
  set}~$S$, meaning a submonoid of~$(R\nnz, \cdot)$. Clearly we must
stipulate $0 \notin S$, otherwise the localization is the zero
ring. If~$R$ is a domain, one might want to take~$S = R\nnz$. While
this is always possible in the commutative setting, one needs an
additional condition if~$R$ is not commutative.

Indeed, let us strive for a localization~$\loc{S}{R}$ on the left,
meaning all elements have the form~$s^{-1}r$ with~$s \in S$ and~$r \in
R$. Since~$s^{-1}, r \in \loc{S}{R}$ this must also be possible for
the reverse product so that~$rs^{-1} = \tilde{s}^{-1} \tilde{r}$ for
some~$\tilde{s} \in S$ and~$\tilde{r} \in R$. Multiplying out, we get
the necessary condition
\begin{equation}
  \label{eq:ore-cond}
  Sr \cap Rs \ne \emptyset \qquad \text{for all $r \in R$ and $s \in
    S$},
\end{equation}
known as the \emph{left Ore condition}; the set~$S$ is then called
\emph{left permutable}.

If~$R$ is a domain, this condition is actually sufficient for
guaranteeing the existence of a unique localization~$\loc{S}{R}
\supseteq R$, which can be constructed essentially as in the
commutative case. However, if~$R$ has \emph{zero divisors}, in general
one does not get an embedding~$R \subseteq \loc{S}{R}$. In this case,
the \emph{extension}~$\epsilon\colon R \rightarrow \loc{S}{R}$ is a
ring homomorphism that is not injective (and of course not
surjective). Its kernel contains at least those~$r \in R$ that
yield~$sr = 0$ for some~$s \in S$ since this implies~$\epsilon(s) \,
\epsilon(r) = 0$ and hence~$\epsilon(r) = 0$, due to~$\epsilon(s)$
being invertible in~$\loc{S}{R}$. In the classical
localization~$\loc{S}{R}$, the kernel should be optimal in the sense
that it contains no other elements than these necessary ones.

\begin{definition}
  \label{def:localization}
  Let~$R$ be an arbitrary ring with~$S \subseteq
  R$. Then~$\epsilon\colon R \rightarrow \loc{S}{R}$ is called a left
  \emph{ring of fractions} if
  \begin{enumerate}
    \setlength{\itemsep}{0pt}
  \item[(a)] all elements~$\epsilon(s)$ with~$s \in S$ are invertible
    in~$\loc{S}{R}$,
  \item[(b)] every element of~$\loc{S}{R}$ has the
    form~$\epsilon(s)^{-1} \epsilon(r)$ for some~$s \in S$, $r \in
    R$,
  \item[(c)] and the kernel of~$\epsilon$ is given by $\{ r \in R \mid
    sr = 0 \; \text{for some $s \in S$} \}$.
  \end{enumerate}
  The ring homomorphism~$\epsilon$ is called the \emph{extension}.
\end{definition}

The missing condition is now easy to establish. Assume~$s \in S$ is a
right zero divisor so that~$rs=0$ for some~$r \in R$. Then
also~$\epsilon(r) \, \epsilon(s) = 0$ and hence~$\epsilon(r) = 0$
since~$\epsilon(s)$ is invertible in~$\loc{S}{R}$. But this
implies~$\tilde{s}r = 0$ for some~$\tilde{s} \in S$ by item~(c) of
Definition~\ref{def:localization}. Accordingly, one calls a set~$S$
with the property
\begin{equation}
  \label{eq:reversible}
  (\forall r \in R) \; (0 \in rS \;\Rightarrow\; 0 \in Sr)
\end{equation}
\emph{left reversible}. Together with the left Ore
condition~\eqref{eq:ore-cond}, this turns out to be sufficient for the
existence of a left ring of fractions.

\begin{theorem}
  \label{thm:ore-localization}
  Let~$R$ be an arbitrary ring. Then for any~$S \subseteq R$, the left
  ring of fractions~$\loc{S}{R}$ exists iff~$S$ is multiplicative,
  left permutable and left reversible.
\end{theorem}
\begin{proof}
  See~\cite[Thm.~10.6]{Lam1999}.
\end{proof}

The setting would become much nicer when the extension~$\epsilon\colon
R \rightarrow \loc{S}{R}$ is injective so that we can regard~$R
\subseteq \loc{S}{R}$ as an \emph{embedding}. Unfortunately, the
localization of~$\intdiffop[\Phi]$ that we will work out in the sequel
is not of this type. In fact, one can easily show that the injectivity
of~$\epsilon$ is equivalent to having only regular elements
in~$S$. Following~\cite[\S5.1]{Cohn2000}, we call an element $s \in S$
regular if it is both left and right regular, where left regular means
$rs=0$ implies $r=0$ for all~$r \in R$ while right regular means
$sr=0$ implies $r=0$ for all~$r \in R$. As~$\intdiffop[\Phi]$ contains
plenty of zero divisors, it is hard to achieve a regular denominator
set~$S$.

Of course there are analogous definitions for the \emph{right ring of
  fractions} $\rloc{S}{R}$, but in general the existence
of~$\loc{S}{R}$ does not imply the existence of~$\rloc{S}{R}$ or vice
versa, and even when both exist they need not be isomorphic. In fact,
we shall be dealing with a case that is left permutable
(Lemma~\ref{lem:prob-left-permutable}) but not right permutable
(Proposition~\ref{prop:not-right-permutable}).

The localization that we shall construct is based on the
monoid~$\regprob$ of regular boundary problems over an
integro-differential algebra~$(\galg, \der, \cum)$ with umbral
character set~$\Phi$. The ring~$R$ to be localized is the \emph{ring
  of boundary problems}~$K \regprob$, which is a monoid algebra. For
such cases, the above setting can be somewhat simplified.

A monoid~$S$ is called \emph{left permutable} if it satisfies the Ore
condition~\eqref{eq:ore-cond} with respect to itself, meaning~$Ss \cap
S\tilde{s} \ne \emptyset$ for all~$s, \tilde{s} \in S$. In the absence
of addition, the analog of condition~\eqref{eq:reversible} says that
for all~$s, s_1, s_2 \in S$ with $s_1 s = s_2 s$ there must
exist~$\tilde{s} \in S$ with~$\tilde{s} s_1 = \tilde{s} s_2$; in this
case we call~$S$ \emph{left reversible} (as a monoid). We call~$S$ a
\emph{left Ore monoid} if~$S$ is both left permutable and left
reversible; cf. also~\cite[Def.~1.3]{Picavet2003}.

The following lemma is a special case of~\cite[Lem.~6.6]{Skoda2006}
but we include its proof since it dispenses with the technical
machinery used there. It tells us that we can \emph{transfer left
  permutability} from~$S$ to the monoid algebra $KS$.

\begin{lemma}
  \label{lem:per-from-mon}
  If $S$ is a left permutable monoid, then~$S$ is a left permutable
  subset of~$KS$. Similarly, if $S$ is a left reversible monoid,
  then~$S$ is a left reversible subset of~$KS$.
\end{lemma}
\begin{proof}
  Given $s_i \in S$, $\lambda_i \in K$ and $s \in S$, we have to find
  $\tilde{s}_i \in S$, $\tilde{\lambda}_i \in K$ and $\tilde{s} \in S$
  such that
  \begin{equation}
    \label{eq:ore-cond-mon}
    \tilde{s} (\lambda_1 s_1 + \dots + \lambda_n s_n) = (\tilde{\lambda}_1
    \tilde{s}_1 + \dots + \tilde{\lambda}_n \tilde{s}_n) s.
  \end{equation}
  Using the Ore condition in $S$, we can successively find $\tilde{l}_n,
  \dots, \tilde{l}_1 \in S$ and $\tilde{r}_1, \dots, \tilde{r}_n \in S$ such
  that
  \begin{align*}
    \tilde{l}_n s_1 &= \tilde{r}_1 s\\
    \tilde{l}_{n-1} (\tilde{l}_n s_2) &= \tilde{r}_2 s\\
    \dots &= \dots\\
    \tilde{l}_1 (\tilde{l}_2 \cdots \tilde{l}_n s_n) &=
    \tilde{r}_n s
  \end{align*}
  is fulfilled. Multiplying these equations by $\lambda_1 \,
  \tilde{l}_1 \cdots \tilde{l}_{n-1}$, $\lambda_2 \, \tilde{l}_1
  \cdots \tilde{l}_{n-2}$, $\dots$, $\lambda_n$ on the left yields the
  system
  \begin{align*}
    \tilde{s} (\lambda_1 s_1) &= (\tilde{\lambda}_1 \tilde{s}_1) s\\
    \tilde{s} (\lambda_2 s_2) &= (\tilde{\lambda}_2 \tilde{s}_2) s\\
    \dots &= \dots\\
    \tilde{s} (\lambda_n s_n) &= (\tilde{\lambda}_n \tilde{s}_n) s
  \end{align*}
  if we set $\tilde{s} = \tilde{l}_1 \cdots \tilde{l}_n$ and
  $\tilde{s}_i = \tilde{l}_1 \cdots \tilde{l}_{n-i} \tilde{r}_i$ with
  coefficients $\tilde{\lambda}_i = \lambda_i$. Summing these
  equations gives the desired Ore condition~\eqref{eq:ore-cond-mon}.
  The proof of the second statement follows immediately by induction
  on the number of terms.
\end{proof}

\begin{corollary}
  \label{cor:ore-from-mon}
  If $S$ is a left Ore monoid, the left ring of
  fractions~$\loc{S}{\,(KS)}$ exists.
\end{corollary}
\begin{proof}
  Immediate from Lemma~\ref{lem:per-from-mon} and
  Theorem~\ref{thm:ore-localization}.
\end{proof}

We shall now prove that~$\regprob$ is a left Ore monoid. The
multiplication of boundary problems is realized by the semi-direct
product~\eqref{eq:bp-product} of the multiplicative
monoid $\calg[\der]$ acting on the additive monoid of boundary
spaces. Projecting onto the first factor, it is then clear that the
monic differential operators of~$\calg[\der]$ must form a left Ore
monoid. A coefficient algebra with this property shall be called
\emph{left extensible}.

\begin{example}
  \label{ex:gcd-diffop}
  Left permutability requires \emph{some} common left multiple for
  given differential operators~$T_1, T_2 \in \calg[\der]$. Hence it
  may seem tempting to restrict ourselves to those coefficient
  algebras~$(\calg, \der)$ that even have a \emph{least common left
    multiple}~$\lclm(T_1, T_2)$. But this setting is not suitable for
  our purposes: As far as we know, there are only two natural
  examples: The rational functions~$\calg = K(x)$ allow an adaption of
  the Euclidean algorithm~\cite[p.~23]{Schwarz2008}, and of course one
  may always take~$\galg = K$. The first case yields a differential
  field, which excludes the existence of an integro-differential
  structure~\cite[p.~518]{RosenkranzRegensburger2008a}. The second
  case restricts us to differential operators with constant
  coefficients, which is excessively restrive.
\end{example}

Fortunately it turns out that we do not need least common left
multiples since left extensible coefficient algebras are easy to come
by, as the next lemma shows. In particular, any integro-differential
algebra~$\galg$ allows~$\calg = K[x]$ so that~$\calg[\der] = A_1(K)$
is the \emph{Weyl algebra}. As another example, consider the standard
setting of Example~\ref{ex:standard}, where for~$\calg$ one may take
the larger ring of \emph{analytic functions}~$C^\omega(\R)$.

\begin{proposition}
  \label{prop:left-ext}
  Any left Noetherian differential domain~$(\calg, \der)$ is left
  extensible.
\end{proposition}
\begin{proof}
  If~$\calg$ is a left Noetherian domain, it satisfies left
  permutability by~\cite[Thm.~5.4]{Cohn2000} and is therefore a left
  Ore domain. But~$\calg[\der]$ is a special case of a skew-polynomial
  algebra and therefore inherits the property of being a left Ore
  domain by~\cite[Prop.~5.9]{Cohn2000}. But this is not enough since
  we need the \emph{monic} differential operators of~$\calg[\der]$ to
  form a left Ore monoid.

  Hence assume~$T_1 T = T_2 T$ for some monic differential
  operators~$T_1, T_2, T \in \calg[\der]$. Since~$\calg[\der]$ is a
  domain, $(T_1 - T_2) T = 0$ implies~$T_1 = T_2$ because~$T = 0$ is
  not possible. Hence the monic differential operators
  from~$\calg[\der]$ form a left reversible monoid. But this monoid is
  indeed a left Ore monoid: According to~\cite[Lem.~1.5.1]{Cohn2006},
  the set of all monic polynomials in the skew polynomial
  ring~$\calg[\der]$ is left permutable whenever~$\calg$ is a left
  Noetherian domain.
\end{proof}

Let us remark that the generalization from \emph{least} common left
multiples to common left multiples is crucial even for the basic case
$\calg = K[x]$. For example when~$T_1 = \der + x$ and~$T_2 =
\der^2+ x \der + x +1$, the least common left multiple of~$T_1$
and~$T_2$ in~$K(x)[\der]$ is
\begin{displaymath}
 \der^3+ \tfrac{2 x^2-1}{x} \, \der^2 + (x^2+1+x) \, \der+
 \tfrac{x^2-1+x^3}{x},
\end{displaymath}
and so gives~$T = x \, \der^3+(2 x^2-1) \, \der^2+(x^3+x^2+x) \, \der
+ (x^3+x^2-1)$ when taken in~$\calg[\der]$. This shows that the least
common left multiple of two monic operators in~$\calg[\der]$ need not
be monic. The conventional computation tools for the Weyl algebra are
therefore not directly applicable for computing monic common left
multiples in~$\calg[\der]$, but some recent
methods~\cite{BostanChyzakLiSalvy2011,Tsai2000} can be adapted to this
purpose [A.~Bostan, private communication].

As explained at the beginning of Section~\ref{sec:umbral-bc}, the main
tool for ensuring the left Ore condition in~$\regprob$ is the
embedding of singular boundary problems into regular ones. This will
be achieved in the Regularization Lemma~\ref{lem:regularization}, by
successively embedding the Stieltjes conditions generating a given
boundary space. Hence the crucial step is to embed a \emph{single
  Stieltjes condition}, which we require to be umbral so that we can
build up a polynomial fundamental system.

\begin{lemma}
  \label{lem:one-cond-high}
  Let~$\beta \in \intdiffop[\Phi]$ be an umbral Stieltjes condition
  over a given integro-differential algebra~$\galg$. Then there is
  a~$k \in \N$ such that~$(\der^{k+1}, [\evl, \dots, \evl \der^{k-1},
  \beta])$ is a regular boundary problem.
\end{lemma}
\begin{proof}
  Since~$\beta$ is umbral, there is a minimal mononomial~$x^k$
  with~$\beta(x^k) \ne 0$. Clearly, $u = (1, x/1!, \cdots, x^k/k!)$ is
  a fundamental system for~$\der^{k+1}$. Using the boundary
  conditions~$\gamma = (\evl, \dots, \evl \der^{k-1}, \beta)$, we
  obtain the $(n+1) \times (n+1)$ evaluation matrix
  \begin{equation*}
    \gamma(u) = 
    \begin{pmatrix}
      1 & 0 & \cdots & 0 & 0\\
      0 & 1 & \ddots & \vdots & 0\\
      \vdots & \ddots & \ddots & 0 & \vdots\\
      0 & \cdots & 0 & 1 & 0\\
      0 & \cdots & 0 & 0 & \beta(x^k/k!)
    \end{pmatrix},
  \end{equation*}
  where the off-diagonal bottom entries vanish by the assumption
  on~$x^k$.
\end{proof}

Note that~$k$ in Lemma~\ref{lem:one-cond-high} can also be
\emph{zero}, for example if~$\beta = \evl$. In that case, it is not
necessary to add any initial conditions since the first-order
problem~$(\der, [\evl])$ is already regular.

\begin{example}
  \label{ex:embed-single-cond}
  As a more or less typical case consider~$\beta =
  \evl_1-\evl_0$. Here the minimal mononomial is~$x$, which leads back
  to our standard Example~\ref{ex:second-order}, the regular boundary
  problem~$(\der^2, [\evl_0, \evl_1-\evl_0]) = (\der^2, [\evl_0,
  \evl_1])$. As mentioned earlier, we cannot split off the first-order
  subproblem~$(\der, [\evl_1-\evl_0])$ since it is singular (this is
  an example of a \emph{subproblem that is not a right factor}): If it
  were not, we would know from the Division Lemma~\ref{lem:division}
  that the corresponding unique left factor is~$(\der, [\cum_0^1])$,
  as shown in~\cite[Ex.~28]{RosenkranzRegensburger2008a} and also
  below (Example~\ref{ex:ore-quad}). But of course multiplying out
  just gives the degenerate problem
  \begin{equation*}
    (\der, [\cum_0^1]) \cdot (\der, [\evl_1-\evl_0] = (\der^2,
    [\evl_1-\evl_0])
  \end{equation*}
  and not the desired regularized problem~$(\der^2, [\evl_0, \evl_1])$.
\end{example}

Based on the case of a single boundary condition, we can now embed an
arbitrary boundary problem into a regular one---provided we work with
an umbral character set. This is the subject of the following
\emph{Regularization Lemma}.

\begin{lemma}
  \label{lem:regularization}
  Let~$\Phi$ be an umbral character set for an integro-differential
  algebra~$\galg$. Then for an arbitrary boundary problem~$(T, \bspc)
  \in \allprob$ there is a regular boundary problem~$(S, \mathcal{A})
  \in \regprob$ that has~$(T, \bspc)$ as a subproblem.
\end{lemma}
\begin{proof}
  Let~$(T, \bspc) \in \allprob$ be an arbitrary but fixed boundary
  problem with~$T$ a differential operator of order~$n>0$ and~$\bspc =
  [\beta_1, \dots, \beta_m] \le \galg^*$ of dimension~$m$. We
  write~$\mathcal{J}_n$ for the space of initial conditions~$[\evl,
  \dots, \evl \der^{n-1}]$ and~$\bspc_k$ for the partial boundary
  space~$[\beta_1, \dots, \beta_k]$ with the convention that~$\bspc_0
  = O$. We will now prove that for each~$k = 0, \dots, m$ there is a
  regular boundary problem~$(S, \mathcal{A})$ that has~$(T, \bspc_k)$
  as a subproblem. Taking~$k=m$, the theorem follows.

  We use induction on~$k$. The base case follows by setting~$(S,
  \mathcal{A}) = (T, \mathcal{J}_n)$. For the induction step,
  assume~$(\tilde{S}, \tilde{\mathcal{A}})$ is a regular boundary
  problem that has~$(T, \bspc_{k-1})$ as a subproblem. We have to
  construct a regular problem~$(S, \mathcal{A})$ that has~$(T,
  \bspc_k)$ as a subproblem. Letting~$\tilde{G}$ be the Green's
  operator of~$(\tilde{S}, \tilde{\mathcal{A}})$, assume first
  that~$\beta_k \tilde{G}$ is degenerate. Since~$\beta_k \tilde{G}$
  vanishes on~$\galg$, we obtain
  \begin{equation*}
    \Img{\tilde{G}} = \orth{\tilde{\mathcal{A}}}
    \le \orth{[\beta_k]}
    \qquad\text{and hence}\qquad
    [\beta_k] \le \tilde{\mathcal{A}},
  \end{equation*}
  so we may set~$(S, \mathcal{A}) = (\tilde{S}, \tilde{\mathcal{A}})$
  in that case. Now assume $\beta_k \tilde{G}$ is nondegenerate and
  hence umbral. Lemma~\ref{lem:one-cond-high} yields a regular
  problem~$(\tilde{T}, \tilde{\bspc}) = (\der^{r+1}, [\evl, \dots,
  \evl \der^{r-1}, \beta_k \tilde{G}])$. We define the boundary
  problem
  \begin{align*}
    (S, \mathcal{A}) &= (\tilde{T}, \tilde{\bspc}) \cdot
    (\tilde{S}, \tilde{\mathcal{A}}) = (\tilde{T} \tilde{S}, [\evl
    \tilde{S}, \dots, \evl \der^{r-1} \tilde{S}, \beta_k
    \tilde{G} \tilde{S}] + \tilde{\mathcal{A}})\\
    &= (\tilde{T} \tilde{S}, [\evl \tilde{S}, \dots, \evl
    \der^{r-1} \tilde{S}, \beta_k] + \tilde{\mathcal{A}}),
  \end{align*}
  where the last equality follows since the conditions~$\beta_k u = 0$
  and~$\beta_k \tilde{G} \tilde{S} u = 0$ are equivalent for~$u \in
  \orth{\tilde{\mathcal{A}}}$. The boundary problem $(S, \mathcal{A})$
  is clearly regular since it is the product of two regular
  problems, and it has~$(T, \bspc_k)$ as a subproblem
  because~$\bspc_{k-1} \le \tilde{\mathcal{A}}$ by the induction
  hypothesis.
\end{proof}

We are now ready to prove \emph{left permutability} for the
monoid~$\regprob$ by merging the given factors into a singular problem
that is subsequently embedded into a regular problem by virtue of
Lemma~\ref{lem:regularization}.

\begin{lemma}
  \label{lem:prob-left-permutable}
  Let~$\Phi$ be an umbral character set for an integro-differential
  algebra~$\galg$ with left extensible coefficient
  algebra~$\calg$. Then~$\regprob$ is a left permutable monoid.
\end{lemma}
\begin{proof}
  Given~$(T_1, \bspc_1), (T_2, \bspc_2) \in \regprob$, we must
  find~$(\tilde{T}_1, \tilde{\bspc}_1), (\tilde{T}_2, \tilde{\bspc}_2)
  \in \regprob$ such that~$(\tilde{T}_1, \tilde{\bspc}_1) \cdot (T_1,
  \bspc_1) = (\tilde{T}_2, \tilde{\bspc}_2) \cdot (T_2,
  \bspc_2)$. Since~$\calg$ is left extensible,
  Proposition~\ref{prop:left-ext} yields a common left multiple~$T$
  and cofactors~$\tilde{T}_1$ and~$\tilde{T}_2$ such that~$T =
  \tilde{T}_1 T_1 = \tilde{T_2} T_2$. Now set~$\bspc = \bspc_1 +
  \bspc_2$. By Lemma~\ref{lem:regularization} there is a regular
  boundary problem~$(S, \mathcal{A})$ that has~$(T, \bspc)$ as a
  subproblem. But then the boundary problem~$(S, \mathcal{A})$
  has~$(T_1, \bspc_1)$ and~$(T_2, \bspc_2)$ as regular subproblems,
  and Lemma~\ref{lem:division} yields regular boundary
  problems~$(\tilde{T}_1, \tilde{\bspc}_1)$ and~$(\tilde{T}_2,
  \tilde{\bspc}_2)$ such that
  \begin{equation*}
    (S, \mathcal{A}) = (\tilde{T}_1, \tilde{\bspc}_1) \cdot
    (T_1, \bspc_1) = (\tilde{T}_2, \tilde{\bspc}_2) \cdot
    (T_2, \bspc_2)
  \end{equation*}
  as claimed.
\end{proof}

\begin{remark}
  Lemma~\ref{lem:prob-left-permutable} is also true if one
  replaces~$\regprob$ by the monoid of well-posed boundary problems
  defined before Eq.~\eqref{eq:proj-times-fri}. This can be seen
  readily by inspecting the above proof (and the proofs of
  Lemma~\ref{lem:division} and~\ref{lem:regularization}).
\end{remark}

Let us give a simple example of a nontrivial \emph{Ore
  quadruple}~$(T_1, \bspc_1)$, $(T_2, \bspc_2)$, $(\tilde{T_1},
\tilde{\bspc_1})$, $(\tilde{T_2}, \bspc_2)$ that will also serve a
good purpose later on.

\begin{example}
  \label{ex:ore-quad}
  In the standard setting of Example~\ref{ex:standard}, consider the
  two simplest first-order problems on~$[0,1]$, namely
  \begin{equation*}
    (T_1, \bspc_1) = (\der, [\evl_0])
    \qquad\text{and}\qquad
    (T_2, \bspc_2) = (\der, [\evl_1]).
  \end{equation*}
  In that case, we have of course~$T = \der$, and we apply
  Lemma~\ref{lem:regularization} to the boundary problem~$(S,
  \mathcal{A}) = (\der, [\evl_0, \evl_1])$.  We end up with
  \begin{equation*}
    (\der, [\cum_0^1]) \cdot (\der, [\evl_0]) =(\der^2, [\evl_0,
    \evl_1]),
  \end{equation*}
  which is already regular. The corresponding cofactors are then
  \begin{equation*}
    (\tilde{T}_1, \tilde{\bspc}_1) = (\der, [\cum_0^1])
    \qquad\text{and also}\qquad
    (\tilde{T}_2, \tilde{\bspc}_2) = (\der, [\cum_0^1])
  \end{equation*}
  since clearly~$(\der^2, [\cum_0^1 \der, \evl_0]) = (\der^2,
  [\cum_0^1 \der, \evl_1]) = (\der^2, [\evl_0, \evl_1])$.
\end{example}

As mentioned before, $\regprob$ is left permutable but \emph{not right
  permutable}. This means there are boundary problems that do not have
a common right factor. Actually, more is true: Even if we start from
two distinct problems with the same differential operator, \emph{any}
common right multiple comes from a singular factor.

\begin{proposition}
  \label{prop:not-right-permutable}
  Let~$\Phi$ be an arbitrary character set for an integro-differential
  algebra~$\galg$ with coefficient algebra~$\calg$. Assume~$(T,
  \bspc_1), (T, \bspc_2) \in \regprob$ have a common right multiple
  \begin{equation}
    \label{eq:try-right-mult}
    (T, \bspc_1) (S, \cspc_1) = (T, \bspc_2) (S, \cspc_2)
  \end{equation}
  for some right factors~$(S, \cspc_1), (S, \cspc_2) \in
  \allprob$. Then both~$(S, \cspc_1)$ and~$(S, \cspc_2)$ are singular
  whenever~$\bspc_1 \ne \bspc_2$.
\end{proposition}
\begin{proof}
  Assume~$\bspc_1 \ne \bspc_2$. Projecting onto the boundary spaces,
  we have~$\bspc_1 S + \cspc_1 = \bspc_2 S + \cspc_2$. If~$\bspc_1 =
  O$ or~$\bspc_2 = O$, we have~$T=1$, so in fact~$\bspc_1 =
  \bspc_2$. Choosing~$\beta_1 \in \bspc_1 \setminus \bspc_2$ we have
  $\beta_1 S \in \bspc_2 S + \mathcal{C}_2$ but~$\beta_1 S \notin
  \bspc_2 S$. Hence we can write~$\beta_1 S = \beta_2 S + \gamma_2$
  for some~$\beta_2 \in \bspc_2$ and nonzero~$\gamma_2 \in
  \cspc_2$. But then~$\gamma_2 = (\beta_1 - \beta_2) S \in
  \orth{\Ker{S}} \cap \cspc_2$ implies that~$(S, \cspc_2)$ is
  singular. By symmetry, we see that~$(S, \cspc_2)$ is singular as
  well.
\end{proof}

Having established that~$\regprob$ is left permutable, the only thing
missing for the localization is left reversibility, which is very easy
in our case. Hence we obtain the desired \emph{left ring of fractions}.

\begin{theorem}
  \label{thm:localization}
  Let~$\Phi$ be an umbral character set for the integro-differential
  algebra $(\galg, \der, \cum)$ with left extensible coefficient
  algebra~$\calg$. Then there exists a left ring of
  fractions~$\locprob$ of the ring of boundary problems~$K \regprob$
  with denominator set~$\regprob$.
\end{theorem}
\begin{proof}
  By Corollary~\ref{cor:ore-from-mon} it suffices to check
  that~$\regprob$ is an Ore monoid. Since left permutability
  of~$\regprob$ has been shown in
  Lemma~\ref{lem:prob-left-permutable}, it remains to show
  that~$\regprob$ is a left reversible monoid. So assume we
  have~$(T_1, \bspc_1) (T, \bspc) = (T_2, \bspc_2) (T, \bspc)$ for
  some regular problems~$(T_1, \bspc_1), (T_2, \bspc_2), (T,
  \bspc)$. By left extensibility of~$\calg$ we have~$\tilde{T} T_1 =
  \tilde{T} T_2$ for some monic~$\tilde{T} \in \calg[\der]$ of
  order~$n$. Let~$\tilde{\bspc} = [\evl, \cdots, \evl \der^{n-1}]$ be
  the corresponding space of initial conditions so that~$(\tilde{T},
  \tilde{\bspc})$ is a regular problem. Then we obtain the regular
  product problem
  \begin{equation*}
    \left( (\tilde{T}, \tilde{\bspc}) (T_1, \bspc_1) \right)
    (T, \bspc) = \left( (\tilde{T}, \tilde{\bspc}) (T_2,
      \tilde{\bspc}) \right) (T, \bspc),
  \end{equation*}
  where the two parenthesized factors have the same differential
  operator by the choice of~$\tilde{T}$. Since~$(T, \bspc)$ is a
  regular right factor of this problem, the Division
  Lemma~\ref{lem:division} implies that~$(\tilde{T}, \tilde{\bspc})
  (T_1, \bspc_1) = (\tilde{T}, \tilde{\bspc}) (T_2,
  \tilde{\bspc})$. We conclude that the monic differential operators
  of~$\calg[\der]$ are indeed left reversible.
\end{proof}

The fraction ring~$\locprob$ shall be called the \emph{ring of
  methorious operators} (from the Greek word \greek{μεθόριος} meaning
``making up a boundary''). Note that it exists in particular in the
smooth setting: If one chooses~$\galg = C^\infty(\R)$, every character
is umbral (Proposition~\ref{prop:cinf-weierstrass}), and any
Noetherian domain can be used for the coefficient algebra
(Proposition~\ref{prop:left-ext}), in particular~$\calg = K[x]$
or~$\calg = C^\omega(\R)$. In the sequel, we shall confine ourselves
to the latter setting.

As already observed after Theorem~\ref{thm:ore-localization}, we must
expect that~$\epsilon\colon K \regprob \rightarrow \locprob$ \emph{is
  not an embedding}. For example, we have $N = (\der, [\evl_0]) -
(\der, [\evl_1]) \in \Ker{\epsilon}$ in the standard setting of~$\galg
= C^\infty(\R)$. For checking this, we use the characterization
of~$\Ker{\epsilon}$ given in Definition~\ref{def:localization}: It
suffices to find a regular problem that annihilates~$N$ from the
left. Indeed, we have~$(\der, [\cum_0^1]) N = 0$ since we know from
Example~\ref{ex:ore-quad} that
\begin{equation*}
  (\der, [\cum_0^1]) (\der, [\evl_0]) = (\der^2, [\evl_0, \evl_1]) =
  (\der, [\cum_0^1]) (\der, [\evl_1]) .
\end{equation*}
At the moment we do not know~$\Ker{\epsilon}$ in explicit
form. However, we have the following conjecture.

\begin{conjecture}
  \label{conj:extension-kernel}
  Let~$\Phi$ be an umbral character set for an integro-differential
  algebra~$\galg$ with left extensible coefficient algebra~$\calg$,
  and let~$\epsilon\colon K \regprob \rightarrow \locprob$ be the
  extension into the ring of methorious operators. Then we have
  $\sum_i \lambda_i \, (T_i, \bspc_i) \in \Ker{\epsilon}$ iff $\sum_i
  \lambda_i G_i \in (\Phi)$, where $G_i$ is the Green's operator
  of~$(T_i, \bspc_i)$.
\end{conjecture}

The previous example is a case in point. The \emph{intuitive reason}
for our conjecture is this: Generically, a linear combination of
Green's operators has a finite-dimensional cokernel (meaning its image
is annihilated by just finitely many functionals---in the case of a
single operator these functionals are the given boundary
conditions). Such an operator is in some sense ``almost
invertible''. But if a linear combination degenerates into an element
of~$(\Phi)$, its image becomes one-dimensional, and we cannot expect
to invert such an ``operator''.

\section{The Module of Methorious Functions}
\label{sec:methorious-functions}

Extending an operator ring is much more useful if its elements may
still be viewed as \emph{operating} on some---presumably
extended---domain of ``functions''. As explained in the Introduction,
the operational calculus of Mikusi{\'n}ski avoids this problem by
merging ``operators'' and ``operands'' in the Mikusi{\'n}ski field. In
contrast, we shall follow the algebraic analysis approach of
keeping operators and operands in separate structures.

To this end we construct a suitable \emph{module of fractions} that
extends a given integro-differential algebra~$(\galg, \der, \cum)$. In
analogy to the Mikusi{\'n}ski approach we would have to start from a
ring~$R$ of integral operators acting on~$\galg$, construct its
localization~$R^\star$ and let it act on the corresponding
localization~$\galg^\star$. However, this does not work for the
following reason: The natural candidate for~$R$ would be the
$K$-algebra of Green's operators for regular boundary problems (since
this is the only ring over which we have sufficient control). But this
monoid/ring is dual to the regular boundary problems, as we know
from~\eqref{eq:anti-isomorphism}. Since~$\regprob$ is a left Ore
monoid but not a right Ore monoid, $R$ is a right but not a left Ore
ring. Hence we can only create a \emph{right} ring of
fractions~$R^\star$ and correspondigly only a \emph{right} module of
fractions that should extend~$\galg$. But the given
integro-differential algebra~$\galg$ only has a natural \emph{left}
action of~$R \subseteq \intdiffop[\Phi]$.

In fact, it is not quite true that one cannot construct a left module
of fractions from a given left module~$M$ over a right Ore
ring~$R$. If $R^*$ is the right ring of fractions, one may of course
construct the \emph{usual scalar extension} $M^\star = R^\star
\otimes_R M$. One may then refer to~$M^\star$ as a left module of
fractions~\cite[Prop.~7.2]{Skoda2006}. But the problem with this
construction is that we do not know anything about its structure. For
example, there is no useful characterization of the kernel of~$M
\rightarrow M^\star$ as in the Ore construction of
Theorem~\ref{thm:fraction-module} below.

We must therefore take the \emph{differential operators} as our
starting point, and this is why we have constructed the
localization~$\locprob$ of the ring of boundary
problems~$K\regprob$. One might be tempted to take just the simpler
ring~$\calg[\der]$ instead of the unwieldy monoid
ring~$K\regprob$. But this would be too simplistic: In that case one
gets a two-sided inverse~$\der^{-1}$ of the differential operator, no
new functions are generated, and of course we cannot tackle boundary
problems in such a setting. It is essential to work with the richer
structure~$K\regprob$, for which we will have to set up a suitable
left action on~$\galg$.

But first let us briefly review the general setting for the
\emph{localization of modules}. As in
Theorem~\ref{thm:ore-localization}, we go back to the setting of
general rings~$R$. In this case we can construct the module of
fractions in pretty much the same way as the ring of
fractions~\cite[Cor.~II.3.3]{Stenstrom1975}.

\begin{theorem}
  \label{thm:fraction-module}
  Let~$M$ be a left $R$-module, and let~$S \subseteq R$ be a
  multiplicative, right permutable and right reversible denominator
  set~$S \subseteq R$. Then there exists a left
  $\loc{S}{R}$-module~$\loc{S}{M}$, and the kernel of the
  extension~$\mu\colon M \rightarrow \loc{S}{M}$ consists of those $u
  \in M$ for which there exists an~$s \in S$ with~$su=0$.
\end{theorem}

Now let~$(\galg, \der, \cum)$ be an integro-differential algebra
and~$\intdiffop[\Phi]$ its ring of operators. For setting up a
suitable left action of~$\regprob$ on~$\galg$, we recall the
\emph{fundamental formula}~\cite{Mikusinski1959} of the Mikusi{\'n}ski
calculus
\begin{equation}
  \label{eq:fundamental-formula}
  sf = f' + f(0) \, \delta_0,
\end{equation}
where~$s$ denotes the inverse of the standard convolution
operator~$l$, defined by~$lf = \cum_0^x f(\xi) \, d\xi$,
and~$\delta_0$ is simply the multiplicative identity generated by the
construction of the Mikusi{\'n}ski field. Of course one may
write~$\delta_0 = 1$ but we keep the explicit notation for intuitive
reasons since we can interpret~\eqref{eq:fundamental-formula} as the
distributional derivative of a function~$f\colon \R \rightarrow \R$
that is continuous except for a jump at~$0$. In that case, $\delta_0$
is the Dirac distribution concentrated at~$0$. At any rate, the
intuitive content of~\eqref{eq:fundamental-formula} is that~$s$ is a
kind of enhanced differential operator that memorizes the function
value at~$0$ that is otherwise lost. Taking this clue, we define the
action of~$\regprob$ on~$\galg$ by
\begin{equation}
  \label{eq:problem-action1}
    (T, \bspc) \cdot f = T\!f + P\!f \: (T, \bspc),
\end{equation}
where~$P$ is the projector onto~$\Ker{T}$ along~$\orth{\bspc}$. In the
special case of~$T = \der$, this
recovers~\eqref{eq:fundamental-formula} if we think of~$(\der,
[\evl_0])$ as some kind of algebraic representation of the Dirac
distribution~$\delta_0$.

The definition~\eqref{eq:problem-action1} has the consequence that we
must extend~$\galg$ so that it contains the new elements~$(T, \bspc)$,
prior to the extension effected by the localization. This is somewhat
\emph{unsatisfying} but appears to be inevitable in view of the facts
pointed out at the opening of the section. Perhaps in the future one
can find a more powerful localization that generates at once the
extension objects needed in the fundamental
formula~\eqref{eq:fundamental-formula}. At any rate, we must now also
define the action of~$\regprob$ on the new elements, but this is
clearly
\begin{equation}
  \label{eq:problem-action2}
  (\tilde{T}, \tilde{\bspc}) \cdot f \: (T, \bspc) = 
  f \: (\tilde{T} T, \tilde{\bspc} T + \bspc),
\end{equation}
where the boundary problem on the right side is the usual
product~\eqref{eq:bp-product} in the problem monoid~$\regprob$. This
will ensure that the action is well-defined.

The \emph{intuitive meaning} of the ideal element~$f \: (T, \bspc)$ is
that it records various integration constants that were lost under
differentation. When the element is created
by~\eqref{eq:problem-action1}, it encodes all the boundary
values~$\beta(f)$ for~$\beta \in \bspc$, and the function $f$ is
confined to the kernel of~$T$. Further action by a boundary
problem~$(\tilde{T}, \tilde{\bspc})$ according
to~\eqref{eq:problem-action2} leads to a new differential
operator~$\tilde{T} T$ while leaving~$f$ intact. Hence we see that the
function inside an ideal element is always in the kernel of the
associated differential operator. We may therefore restrict the ideal
elements~$f \: (T, \bspc)$ to those with~$Tf=0$. After the action
of~\eqref{eq:problem-action2}, the new boundary space
is~$\tilde{\bspc} T + \bspc$, so the old boundary conditions~$\beta
\in \bspc$ are retained (the ones for which we have boundary
values). The new boundary conditions~$\tilde{\beta} T$
for~$\tilde{\beta} \in \tilde{\bspc}$ give zero on~$f$ since actually
$Tf=0$. In some sense, we have added redundant boundary data.

We must therefore \emph{regulate redundancy} in the ideal
elements. This can be achieved by declaring that~$f \: (T, \bspc)$
should be the same as~$\tilde{G}f \: (T, \bspc)(\tilde{T},
\tilde{\bspc})$ for any~$(\tilde{T}, \tilde{\bspc}) \in \regprob$ with
Green's operator~$\tilde{G}$. This can be understood as follows. The
boundary conditions in~$\tilde{G}f \: (T, \bspc)(\tilde{T},
\tilde{\bspc})$ are~$\tilde{\beta} \in \tilde{\bspc}$ and~$\beta
\tilde{T}$ with~$\beta \in \bspc$. The former yield zero
on~$\tilde{G}f$ since~$\tilde{G}$ is the Green's operator
of~$(\tilde{T}, \tilde{\bspc})$, the latter give back~$\beta \tilde{T}
(\tilde{G} f) = \beta(f)$ just as for the ideal element~$f \: (T,
\bspc)$. This leads to the following definition.

\begin{definition}
  Let~$\mathcal{I}$ be the subspace of~$\galg \otimes_K K\regprob $
  generated by the ideal elements~$f \otimes (T, \bspc) \equiv f \:
  (T, \bspc)$ with~$Tf=0$. Furthermore, let~$\mathcal{I}_0$ be the
  subspace of~$\mathcal{I}$ generated by the elements
  \begin{equation*}
    f \: (T, \bspc) - \tilde{G}f \: (T\tilde{T}, \bspc \tilde{T} +
    \tilde{\bspc}).
  \end{equation*}
  Then we define the \emph{module of methorious functions} to
  be~$\hypfun{\galg} = \galg \oplus \mathcal{I}/\mathcal{I}_0$.
\end{definition}
Let us now check that~$\hypfun{\galg}$ is indeed a module.

\begin{proposition}
  \label{prop:prob-action}
  Let~$(\galg, \der, \cum)$ be an integro-differential algebra with
  character set~$\Phi$. The definitions~\eqref{eq:problem-action1}
  and~\eqref{eq:problem-action2} induce a monoid action of~$\regprob$
  on~$\hypfun{\galg}$ such that it becomes a $K\regprob$-module.
\end{proposition}
\begin{proof}
  Any monoid~$E$ acting on a~$K$-algebra~$A$ extends to an action of
  the monoid ring~$K[E]$ that makes~$A$ into a $K[E]$-module. Hence it
  suffices to verify the statement about the monoid action.

  First of all we must check that~\eqref{eq:problem-action2} does not
  depend on the representative~$f \: (T, \bspc) \in \mathcal{I}$,
  meaning it maps the subspace~$\mathcal{I}_0$ into itself. This
  follows immediately from the associativity of the
  multiplication~\eqref{eq:bp-product} in~$\regprob$.

  Clearly we have~$(1,0) \cdot f = f + 0 \: (1,0) = f$ and $(1,0)
  \cdot (T, \bspc) \: f = (T, \bspc) \: f$, so the unit element is
  respected. Associativity is immediate in the case
  of~\eqref{eq:problem-action2}, so it remains to check that~$(T,
  \bspc) \cdot ((\tilde{T}, \tilde{\bspc}) \cdot f)$ and~$((T,
  \bspc)(\tilde{T}, \tilde{\bspc})) \cdot f$ are equal. The former is
  \begin{equation*}
    (T, \bspc) \cdot \Big( \tilde{T} f + \tilde{P} \! f \: (\tilde{T},
    \tilde{\bspc}) \Big) =
    T \tilde{T} f + P\tilde{T} \! f \: (T, \bspc) + \tilde{P} \! f \:
    (T \tilde{T}, \bspc \tilde{T} + \tilde{\bspc}),
  \end{equation*}
  where~$P$ is the projector onto~$\Ker{T}$ along~$\orth{\bspc}$
  and~$\tilde{P}$ is correspondingly the projector
  onto~$\Ker{\tilde{T}}$ along~$\orth{\tilde{\bspc}}$. If~$G$
  and~$\tilde{G}$ are the Green's operators of~$(T, \bspc)$
  and~$(\tilde{T}, \tilde{\bspc})$, respectively, one can easily check
  that the projector associated with the composite problem
  \begin{equation*}
    (T, \bspc)(\tilde{T}, \tilde{\bspc}) = (T \tilde{T}, \bspc
    \tilde{T} + \tilde{\bspc})
  \end{equation*}
  is~$\tilde{P} + \tilde{G} P \tilde{T}$, meaning it projects
  onto~$\Ker{T \tilde{T}}$ along~$\orth{(\bspc \tilde{T} +
    \tilde{\bspc})}$. Hence we obtain for the other side of the
  prospective equality
  \begin{align*}
    & (T \tilde{T}, \bspc \tilde{T} + \tilde{\bspc}) \cdot f = T
    \tilde{T} f + (\tilde{P} \! f + \tilde{G} P \tilde{T} \! f) \;
    (T \tilde{T}, \bspc \tilde{T} + \tilde{\bspc})\\
    & = T \tilde{T} f + \tilde{G} P \tilde{T} \! f \: (T \tilde{T},
    \bspc \tilde{T} + \tilde{\bspc}) + \tilde{P} \! f \: (T \tilde{T},
    \bspc \tilde{T} + \tilde{\bspc}),
  \end{align*}
  where the first and the last term is identical to the corresponding
  terms on the left hand side while the middle terms are equal since
  their difference is in~$\mathcal{I}_0$.
\end{proof}

Since~$\hypfun{\galg}$ is a~$K\regprob$-module, we obtain
the~$\locprob$-module~$\locfun{\galg}$ of \emph{methorious
  hyperfunctions} by localization via
Theorem~\ref{thm:fraction-module}. We cannot expect the
extension~$\mu\colon \hypfun{\galg} \rightarrow \locfun{\galg}$ to be
injective since its kernel
\begin{equation}
  \label{eq:kernel-hypfun}
  \mathcal{K} = \{ \phi \in \hypfun{\galg} \mid \exists (T, \bspc) \in
  \regprob\colon (T, \bspc) \cdot \phi = 0 \}
\end{equation}
contains elements like~$(\der, [\evl_1]) - (\der, [\evl_0])$, which
is annihilated upon multiplying with~$(\der, [\cum_0^1])$ from the left
(see the example before Conjecture~\ref{conj:extension-kernel}). But
fortunately no elements of~$\galg$ are lost.

\begin{proposition}
  \label{prop:hypfun-extension}
  Let~$(\galg, \der, \cum)$ be an integro-differential algebra with
  character set~$\Phi$. Then we have an embedding~$\galg \subset
  \locfun{\galg}$.
\end{proposition}
\begin{proof}
  Assume~$(T, \bspc) \cdot f = 0$ for some~$f \in \galg$ and~$(T,
  \bspc) \in \regprob$. Then~$Tf + P \! f \: (T, \bspc) = 0$ implies
  that $Tf=0$ and~$Pf=0$. But the former means that~$f \in \Ker{T}$
  and the latter that~$f \in \orth{\bspc}$. Since~$(T, \bspc)$ is a
  regular problem, we have a direct sum~$\Ker{T} \dirs \orth{\bspc} =
  \galg$ and so~$f=0$.
\end{proof}

In the module of methorious hyperfunctions, we can finally justify our
earlier notation~$(T, \bspc)^{-1}$ for the \emph{Green's operator} of
a regular boundary problem~$(T, \bspc) \in \regprob$. To avoid
confusion with~$(T, \bspc)^{-1} \in \locprob$ we refrain from the
reciprocal notation for Green's operators in the scope of the
following theorem. But the result of the theorem is of course that it
anyway does not matter how we interpret~$(T, \bspc)^{-1} f$ since it
amounts to the same.

\begin{proposition}
  \label{prop:greensop-localization}
  We have~$(T, \bspc)^{-1} \cdot f = Gf$ for all~$f \in
  \galg$. Moreover, if~$Tf=0$ then~$(T, \bspc)^{-1} \cdot
  f \: (T, \bspc) = f$.
\end{proposition}
\begin{proof}
  We have~$(T, \bspc) \cdot Gf = TGf + PGf \: (T, \bspc) = f$
  since~$TG=1$ and~$PG=0$. Multiplying by~$(T,\bspc)^{-1} \in
  \locprob$ yields the first result claimed. Now assume~$Tf=0$. We
  obtain~$(T, \bspc) \cdot f = 0 + P \! f \: (T, \bspc) = f \: (T,
  \bspc)$ since~$Pf = f$ for~$f \in \Ker{T}$. Again we multiply
  by~$(T,\bspc)^{-1}$ to gain the result.
\end{proof}

We conclude with an example that hints at possible applications of our
noncommutative Mikusi{\'n}ski calculus. The classical Mikusi{\'n}ski
calculus has only one \emph{fundamental formula} since boundary values
(or rather: initial values) are only processed at~$0$. In contrast,
there are plenty of fundamental formulae in the noncommutative
Mikusi{\'n}ski calculus.

\begin{example}
  \label{ex:fundamental-formulae}
  Let us write~$\der_\xi$ and~$\delta_\xi$ as abbreviations for the
  problems~$(\der, [\evl_\xi]) \in \locprob$ and~$(\der, [\evl_\xi])
  \in \locfun{\galg}$, respectively. Then we have~$\der_\xi f = f' +
  f(\xi) \, \delta_\xi$ by the definition of the
  action~\eqref{eq:problem-action1}. So we have algebraic
  representations for all the Dirac distributions. But there are other
  methorious functions that do not have any distributional
  counterpart. For example, let us consider~$\der_F = (\der,
  [\cum_0^1]) \in \locprob$ and~$\epsilon = (\der, [\cum_0^1]) \in
  \locfun{\galg}$. This yields the fundamental formula
  \begin{equation*}
    \der_F f = f' + \left(\int_0^1 f(\xi) \, d\xi \right) \epsilon,
  \end{equation*}
  so~$\epsilon$ is a kind of ``smeared out'' distribution that keeps
  the mean value of a given function~$f$.
\end{example}

\begin{example}
  \label{ex:inhom-bp}
  Finally let us see how one solves \emph{inhomogeneous boundary
    problems} in the noncommutative Mikusi{\'n}ski calculus. We use
  the standard setting of~$\galg = C^\infty(\R)$ and consider the
  problem
  \begin{equation}
    \label{eq:inhom-bp}
    \bvp{u''=f}{u(0)=a, u(b)=b}
  \end{equation}
  for given boundary values~$a, b \in \R$. We know how to compute the
  Green's operator~$G$ of the homogeneous problem~$(\der^2, [\evl_0,
  \evl_1])$ as mentioned at the end of Section~\ref{sec:monoid-bp}. In
  this case the projector~$P$ onto~$\Ker{\der^2} = [1, x]$
  along~$\orth{[\evl_0, \evl_1]}$ is given by~$Pu = u(0) (1-x) + u(1)
  x$. Hence we have
  \begin{align*}
    (\der^2, [\evl_0, \evl_1]) \, u &= u'' + \Big( u(0) \, (1-x) +
    u(1) \, x \Big) \, (\der^2, [\evl_0, \evl_1])\\
    & = f + \Big( a \, (1-x) + b \, x \Big) (\der^2, [\evl_0,
    \evl_1]),\\[0.5ex]
    u &= Gf + a \, (1-x) + b \, x
  \end{align*}
  by Proposition~\ref{prop:greensop-localization}.
\end{example}

\section{Conclusion}
\label{sec:conclusion}

As already indicated at various points above, our construction has
several \emph{loose ends}. To begin with, it would be preferrable to
localize in a subring (or even all) of~$\intdiffop[\Phi]$ rather than
the monoid ring~$K\regprob$. This would have the advantage that one
has a natural action on the underlying integro-differential
algebra~$(\galg, \der, \cum)$, and the somewhat artificial action on
the module of methorious functions~$\hypfun{\galg}$ would be
unnecessary. Regarding the latter, we have already remarked in
Section~\ref{sec:methorious-functions} that the current two-stage
process of creating the localization~$\locfun{\galg}$ is
unsatisfactory: In Mikusi{\'n}ski's setup, all the ``ideal elements''
like~$s$ and~$\delta_0$ are an immediate result of the localization
while we have to supply them offhand, prior to localization. We would
like to find a better formulation in the future that will avoid this
kind of inadequacy.

One possible path towards such an improved localization is suggested
by the prominent appearance of \emph{singular problems} in the proof
of Lemma~\ref{lem:prob-left-permutable}. It may be worthwhile to
expand the ring~$K\regprob$ to include (all or some) singular boundary
problems. Of course this will also increase the kernel of the
corresponding extension (we cannot expect to invert singular
problems), but perhaps the resulting ring of fractions is more
natural. In particular, it may be possible to interpret the sum of two
(singular) boundary problems in some useful way, in contrast to the
formal sums of~$K\regprob$. The results
of~\cite{KorporalRegensburgerRosenkranz2011,Korporal2012} will be
useful for the work in this direction.

Staying with the current construction, there are some obvious open
questions: Fist of all, the \emph{kernel of the extension}~$\epsilon$
into the methorious operators should be determined, possibly by
proving Cojecture~\ref{conj:extension-kernel}. Likewise, the
kernel~\ref{eq:kernel-hypfun} of the extension~$\mu$ into the
methorious hyperfunctions is to be computed.

Finally, we would like to draw attention to the interesting link
between integro-differential algebras and the umbral calculus
(Section~\ref{sec:umbral-bc}), which deserves to be studied in more
depth. There is also an intriguing relation~\cite{Guo2001} between the
umbral calculus and the Rota-Baxter algebras (both of these being
favourite topic of G.-C.~Rota), which might benefit from the new
perspective afforded by boundary problems in integro-differential
algebras (since the latter are special cases of differential
Rota-Baxter algebras and hence of plain Rota-Baxter algebras). As a
more mundane goal, it would also be important to find a better
characterization of umbral character sets that strengthens the
separativity and completeness conditions given before
Example~\ref{ex:separative}.

\begin{acknowledgements}
  First of all we would like to acknowledge Georg Regensburger, who
  has worked with the first author on an earlier (much more
  specialized and mathematically flawed) version of localization. We
  thank him also for his insight about monic common left multiples and
  the corresponding reference (last paragraph of
  Proposition~\ref{prop:left-ext}). We express our gratitude also to
  A.~Bostan for his pertinent remarks on left common multiples in the
  Weyl algebra (see the remark after the proof of
  Proposition~\ref{prop:left-ext}). Furthermore, we are grateful for
  illuminating discussions with I.~Dimovski and M.~Spiridonova.
\end{acknowledgements}


\begin{thebibliography}{10}
\providecommand{\url}[1]{{#1}}
\providecommand{\urlprefix}{URL }
\expandafter\ifx\csname urlstyle\endcsname\relax
  \providecommand{\doi}[1]{DOI~\discretionary{}{}{}#1}\else
  \providecommand{\doi}{DOI~\discretionary{}{}{}\begingroup
  \urlstyle{rm}\Url}\fi

\bibitem{Agarwal1986}
Agarwal, R.: Boundary value problems for higher order differential equations.
\newblock World Scientific Publishing Co., Teaneck, NJ (1986)

\bibitem{AgarwalORegan2008}
Agarwal, R., O'Regan, D.: An introduction to ordinary differential equations.
\newblock Springer (2008)

\bibitem{AlbrecherConstantinescuPirsicEtAl2009}
Albrecher, H., Constantinescu, C., Pirsic, G., Regensburger, G., Rosenkranz,
  M.: An algebraic approach to the analysis of gerber-shiu functions.
\newblock Insurance: Mathematics and Economics \textbf{Special Issue on
  Gerber-Shiu Functions}, accpeted (2009)

\bibitem{Baxter1960}
Baxter, G.: An analytic problem whose solution follows from a simple algebraic
  identity.
\newblock Pacific J. Math. \textbf{10}, 731--742 (1960)

\bibitem{Berg1962}
Berg, L.: Einf\"uhrung in die {O}peratorenrechung.
\newblock Deutscher Verlag der Wissenschaften, Berlin (1962)

\bibitem{BostanChyzakLiSalvy2011}
Bostan, A., Chyzak, F., Li, Z., Salvy, B.: Fast computation of common left
  multiples of linear ordinary differential operators.
\newblock ACM Commun. Comput. Algebra \textbf{45}(1/2), 111--112 (2011).
\newblock \doi{10.1145/2016567.2016581}.
\newblock \urlprefix\url{http://doi.acm.org/10.1145/2016567.2016581}

\bibitem{Bourbaki1987}
Bourbaki, N.: Topological vector spaces. {C}hapters 1--5.
\newblock Elements of Mathematics (Berlin). Springer, Berlin (1987)

\bibitem{Chyzak1994}
Chyzak, F.: Holonomic systems and automatic proofs of identities.
\newblock Tech. Rep. 2371, Institut National de la R{\'e}cherche en
  Informatique et Automatique (1994)

\bibitem{Cohn2000}
Cohn, P.M.: Introduction to Ring Theory.
\newblock Springer (2000)

\bibitem{Cohn2006}
Cohn, P.M.: Free ideal rings and localization in general rings, \emph{New
  Mathematical Monographs}, vol.~3.
\newblock Cambridge University Press, Cambridge (2006).
\newblock \doi{10.1017/CBO9780511542794}.
\newblock \urlprefix\url{http://dx.doi.org/10.1017/CBO9780511542794}

\bibitem{DiBucchianico1998}
DiBucchianico, A.: An introduction to umbral calculus (1998).
\newblock Lectue notes, available at \url{www.win.tue.nl/~adibucch/eidma.ps}

\bibitem{Dimovski1990}
Dimovski, I.: Convolutional calculus, \emph{Mathematics and its Applications
  (East European Series)}, vol.~43.
\newblock Kluwer Academic Publishers Group, Dordrecht (1990)

\bibitem{Dimovski1994}
Dimovski, I.: Nonlocal operational calculi.
\newblock Trudy Mat. Inst. Steklov. \textbf{Izbran. Voprosy Mat. Fiz. i Anal.
  203}, 58--73 (1994)

\bibitem{Dimovski2012}
Dimovski, I.: Operational calculi for boundary value problems (2012).
\newblock Talk at the Algebraic Aspects of Differential and Integral Operators
  Session (AADIOS'12), 18th Conference on Applications of Computer Algebra
  (ACA), Sofia, 25-28 June

\bibitem{Duffy2001}
Duffy, D.G.: Green's functions with applications.
\newblock Studies in Advanced Mathematics. Chapman \& Hall, Boca Raton, FL
  (2001)

\bibitem{Engl1997}
Engl, H.W.: Integralgleichungen.
\newblock Springer Lehrbuch Mathematik. Springer-Verlag, Vienna (1997)

\bibitem{FloweHarris1993}
Flowe, R.P., Harris, G.A.: A note on generalized {V}andermonde determinants.
\newblock Siam J. Matrix anal. appl \textbf{14}, 1146--1151 (1993)

\bibitem{Grigoriev1990}
Grigoriev, D.Y.: Complexity of factoring and calculating the {GCD} of linear
  ordinary differential operators.
\newblock J. Symbolic Comput. \textbf{10}(1), 7--37 (1990)

\bibitem{Guo2001}
Guo, L.: Baxter algebras and the umbral calculus.
\newblock Adv. in Appl. Math. \textbf{27}(2-3), 405--426 (2001).
\newblock \doi{10.1006/aama.2001.0742}.
\newblock \urlprefix\url{http://dx.doi.org/10.1006/aama.2001.0742}.
\newblock Special issue in honor of Dominique Foata's 65th birthday
  (Philadelphia, PA, 2000)

\bibitem{Guo2002}
Guo, L.: Baxter algebras and differential algebras.
\newblock In: Differential algebra and related topics (Newark, NJ, 2000), pp.
  281--305. World Sci. Publ., River Edge, NJ (2002)

\bibitem{Heaviside1893}
Heaviside, O.: On operators in physical mathematics. part i.
\newblock In: Proceedings of the Royal Society, vol.~52, pp. 504--529. London
  (1893)

\bibitem{Heaviside1894}
Heaviside, O.: On operators in physical mathematics. part ii.
\newblock In: Proceedings of the Royal Society, vol.~54, pp. 105--143. London
  (1894)

\bibitem{Kamke1967}
Kamke, E.: Differentialgleichungen. {L}\"osungsmethoden und {L}\"osungen.
  {T}eil {I}: {G}ew\"ohnliche {D}ifferentialgleichungen, \emph{Mathematik und
  ihre Anwendungen in Physik und Technik {A}}, vol.~18, eighth edn.
\newblock Akademische Verlagsgesellschaft, Leipzig (1967)

\bibitem{Kolchin1973}
Kolchin, E.: Differential algebra and algebraic groups, \emph{Pure and Applied
  Mathematics}, vol.~54.
\newblock Academic Press, New York-London (1973)

\bibitem{Korporal2012}
Korporal, A.: Symbolic methods for generalized green's operators and boundary
  problems.
\newblock Ph.D. thesis, Johannes Kepler University, Linz, Austria (2012).
\newblock In progress

\bibitem{KorporalRegensburgerRosenkranz2010}
Korporal, A., Regensburger, G., Rosenkranz, M.: A {M}aple package for
  integro-differential operators and boundary problems (2010).
\newblock Also presented as a poster at ISSAC '10.

\bibitem{KorporalRegensburgerRosenkranz2011}
Korporal, A., Regensburger, G., Rosenkranz, M.: Regular and singular boundary
  problems in maple.
\newblock In: Proceedings of the 13th International Workshop on Computer
  Algebra in Scientific Computing, CASC'2011 (Kassel, Germany, September 5-9,
  2011), \emph{Lecture Notes in Computer Science}, vol. 6885. Springer (2011)

\bibitem{KorporalRegensburgerRosenkranz2012}
Korporal, A., Regensburger, G., Rosenkranz, M.: Symbolic computation for
  ordinary boundary problems in maple.
\newblock In: Proceedings of the 37th International Symposium on Symbolic and
  Algebraic Computation (ISSAC'09) (2012).
\newblock Software presentation

\bibitem{Krattenthaler1999}
Krattenthaler, C.: Advanced determinant calculus.
\newblock S\'em. Lothar. Combin. \textbf{42}, Art. B42q, 67 pp. (electronic)
  (1999).
\newblock The Andrews Festschrift (Maratea, 1998)

\bibitem{Lam1999}
Lam, T.Y.: Lectures on modules and rings, \emph{Graduate Texts in Mathematics},
  vol. 189.
\newblock Springer-Verlag, New York (1999)

\bibitem{Mikusinski1959}
Mikusi{\'n}ski, J.: Operational calculus, \emph{International Series of
  Monographs on Pure and Applied Mathematics}, vol.~8.
\newblock Pergamon Press, New York (1959)

\bibitem{Nachbin1949}
Nachbin, L.: Sur les alg{\`e}bres denses de fonctions diff\'erentiables sur une
  vari{\'e}t{\'e}.
\newblock C. R. Acad. Sci. \textbf{228}, 1549--1551 (1949)

\bibitem{Picavet2003}
Picavet, G.: Localization with respect to endomorphisms.
\newblock Semigroup Forum \textbf{67}(1), 76--96 (2003).
\newblock \doi{10.1007/s00233-002-0008-2}.
\newblock \urlprefix\url{http://dx.doi.org/10.1007/s00233-002-0008-2}

\bibitem{PutSinger2003}
van~der Put, M., Singer, M.F.: Galois theory of linear differential equations,
  \emph{Grundlehren der Mathematischen Wissenschaften}, vol. 328.
\newblock Springer, Berlin (2003)

\bibitem{RegensburgerRosenkranz2009}
Regensburger, G., Rosenkranz, M.: An algebraic foundation for factoring linear
  boundary problems.
\newblock Ann. Mat. Pura Appl. (4) \textbf{188}(1), 123--151 (2009).
\newblock \doi{10.1007/s10231-008-0068-3}.
\newblock \urlprefix\url{http://dx.doi.org/10.1007/s10231-008-0068-3}.
\newblock DOI:10.1007/s10231-008-0068-3

\bibitem{RegensburgerRosenkranzMiddeke2009}
Regensburger, G., Rosenkranz, M., Middeke, J.: A skew polynomial approach to
  integro-differential operators.
\newblock In: I{SSAC} 2009---{P}roceedings of the 2009 {I}nternational
  {S}ymposium on {S}ymbolic and {A}lgebraic {C}omputation, pp. 287--294. ACM,
  New York (2009)

\bibitem{Ritt1966}
Ritt, J.F.: Differential algebra.
\newblock Dover Publications Inc., New York (1966)

\bibitem{Rosenkranz1997}
Rosenkranz, M.: {Lagrange Inversion}.
\newblock Master's thesis, Research Institute for Symbolic Computation (RISC),
  Johannes Kepler University, Linz (1997).
\newblock Available at
  \url{http://www.risc.jku.at/publications/download/risc_2383/Diploma.pdf}.

\bibitem{Rosenkranz2005}
Rosenkranz, M.: A new symbolic method for solving linear two-point boundary
  value problems on the level of operators.
\newblock {J}. {S}ymbolic {C}omput. \textbf{39}(2), 171--199 (2005)

\bibitem{RosenkranzRegensburger2008}
Rosenkranz, M., Regensburger, G.: Integro-differential polynomials and
  operators.
\newblock In: D.~Jeffrey (ed.) ISSAC'08: Proceedings of the 2008 International
  Symposium on Symbolic and Algebraic Computation. ACM Press (2008)

\bibitem{RosenkranzRegensburger2008a}
Rosenkranz, M., Regensburger, G.: Solving and factoring boundary problems for
  linear ordinary differential equations in differential algebras.
\newblock Journal of Symbolic Computation \textbf{43}(8), 515--544 (2008).
\newblock \doi{10.1016/j.jsc.2007.11.007}

\bibitem{RosenkranzRegensburgerTecBuchberger2009}
Rosenkranz, M., Regensburger, G., Tec, L., Buchberger, B.: A symbolic framework
  for operations on linear boundary problems.
\newblock In: V.P. Gerdt, E.W. Mayr, E.H. Vorozhtsov (eds.) Computer Algebra in
  Scientific Computing. Proceedings of the 11th International Workshop (CASC
  2009), \emph{LNCS}, vol. 5743, pp. 269--283. Springer, Berlin (2009)

\bibitem{RosenkranzRegensburgerTecBuchberger2012}
Rosenkranz, M., Regensburger, G., Tec, L., Buchberger, B.: Symbolic analysis of
  boundary problems: {F}rom rewriting to parametrized {G}r{\"o}bner bases.
\newblock In: U.~Langer, P.~Paule (eds.) Numerical and Symbolic Scientific
  Computing: Progress and Prospects, pp. 273--331. Springer (2012)

\bibitem{Rota1969}
Rota, G.C.: Baxter algebras and combinatorial identities ({I}, {II}).
\newblock Bull. Amer. Math. Soc. \textbf{75}, 325--334 (1969)

\bibitem{SalvyZimmerman1994}
Salvy, B., Zimmerman, P.: Gfun: a maple package for the manipulation of
  generating and holonomic functions in one variable.
\newblock ACM Trans. Math. Softw. \textbf{20}(2), 163--177 (1994).
\newblock \doi{http://doi.acm.org/10.1145/178365.178368}

\bibitem{Schwarz1989}
Schwarz, F.: A factorization algorithm for linear ordinary differential
  equations.
\newblock In: ISSAC '89: Proceedings of the ACM-SIGSAM 1989 international
  symposium on Symbolic and algebraic computation, pp. 17--25. ACM Press, New
  York, NY, USA (1989).
\newblock \doi{http://doi.acm.org/10.1145/74540.74544}

\bibitem{Schwarz2008}
Schwarz, F.: Algorithmic {L}ie theory for solving ordinary differential
  equations, \emph{Pure and Applied Mathematics (Boca Raton)}, vol. 291.
\newblock Chapman \& Hall/CRC, Boca Raton, FL (2008)

\bibitem{Skoda2006}
{\v{S}}koda, Z.: Noncommutative localization in noncommutative geometry.
\newblock In: Non-commutative localization in algebra and topology,
  \emph{London Math. Soc. Lecture Note Ser.}, vol. 330, pp. 220--313. Cambridge
  Univ. Press, Cambridge (2006).
\newblock \doi{10.1017/CBO9780511526381.015}.
\newblock \urlprefix\url{http://dx.doi.org/10.1017/CBO9780511526381.015}

\bibitem{Spiridonova2010}
Spiridonova, M.: Operational methods in the environment of a computer algebra
  system (2010).
\newblock Talk at the Algebraic Aspects of Differential and Integral Operators
  Session (AADIOS'10), 16th Conference on Applications of Computer Algebra
  (ACA), Vlora, 24-27 June

\bibitem{Stakgold1979}
Stakgold, I.: Green's functions and boundary value problems.
\newblock John Wiley \& Sons, New York (1979)

\bibitem{Stenstrom1975}
Stenstr{\"o}m, B.: Rings of quotients.
\newblock Springer-Verlag, New York (1975).
\newblock Die Grundlehren der Mathematischen Wissenschaften, Band 217, An
  introduction to methods of ring theory

\bibitem{Tec2011}
Tec, L.: A symbolic framework for general polynomial domains in theorema:
  Applications to boundary problems.
\newblock Ph.D. thesis, Research Institute for Symbolic Computation, Johannes
  Kepler University, A-4040 Linz, Austria (2011)

\bibitem{Tsai2000}
Tsai, H.: Algorithms for algebraic analysis.
\newblock Ph.D. thesis, University of California at Berkeley (2000)

\bibitem{Tsarev1996}
Tsarev, S.P.: An algorithm for complete enumeration of all factorizations of a
  linear ordinary differential operator.
\newblock In: ISSAC '96: Proceedings of the 1996 international symposium on
  Symbolic and algebraic computation, pp. 226--231. ACM Press, New York, NY,
  USA (1996).
\newblock \doi{http://doi.acm.org/10.1145/236869.237079}

\end{thebibliography}
\end{document}